\documentclass[a4paper,12pt,english,reqno]{smfart}
\usepackage[latin1]{inputenc}
\usepackage{amsmath}
\usepackage{mathabx}
\usepackage{amsfonts}
\usepackage{amssymb}
\usepackage{graphicx}
\usepackage{graphics}
\usepackage{mathabx}
\usepackage{epsfig}
\usepackage{amscd}
\usepackage{color}
\usepackage{shadow}
\usepackage{a4wide}
\usepackage[all]{xy}
\usepackage{latexsym}
\usepackage{hyperref,amsthm}
\usepackage{amsopn}
\usepackage{epic, eepic}
\usepackage{fancyhdr}
\usepackage{here}
\newtheorem{th1}{Theorem}[section]

\newtheorem{lem}[th1]{ Lemma}

\newtheorem{rem}[th1]{ Remark}

\author[{H.Hani and M.Khenissi}]{Houda Hani and Moez Khenissi}
\title[\textbf{On a finite difference scheme for blow up solutions}]{On a finite difference scheme for blow up solutions for the Chipot-Weissler equation}
\begin{document}
	\frontmatter
\begin{abstract}
In this paper, we are interested in the numerical analysis of blow up for the Chipot-Weissler equation $u_{t}=\Delta u+|u|^{p-1}u-|\nabla u|^{q}$ with Dirichlet boundary conditions in bounded domain when $p>1$ and $1\leq q\leq \frac{2p}{p+1}$.\\
To approximate the blow up solution, we construct a finite difference scheme and we prove that the numerical solution satisfies the same properties of the exact one and blows up in finite time.
\end{abstract}

\keywords{
Nonlinear parabolic equation, Chipot-Weissler equation, finite time blow up, finite difference scheme, numerical solution, nonlinear gradient term.
}
\frontmatter
	\maketitle
	\mainmatter\date{}
	\tableofcontents
\section{Introduction}
In this paper, we study the numerical approximation of solutions which achieve blow up in finite time of the Chipot-Weissler equation
\begin{equation}
u_{t}=\Delta u+|u|^{p-1}u-|\nabla u|^{q} \ \text{ in } \ \ \mathbb{R}^{+} \times\Omega,
\label{eq11}
\end{equation}
with Dirichlet boundary conditions
\begin{equation}
u(t,x)=0,  \  t>0 \  \ \text{ and } \ x\in \partial\Omega,
\label{eq12}
\end{equation}
and initial data
\begin{equation}
u(0,x)=u_{0}(x), \ \ x\in \Omega,
\label{eq13}
\end{equation}
where $\Omega$ is a regular bounded domain in $ \mathbb{R}^{d}\ \text{ and } \ p,\ q$ are fixed finite parameters.\\

This problem represents a model in population dynamics which is proposed by Souplet in \cite{souplet}, where \eqref{eq11}-\eqref{eq13} describe the evolution of the population density of a biological species, under the effect of certain natural mechanisms, $u(t,x)$ denotes the spatial density of individuals located near a point $x\in \Omega$ at a time $t\geq 0.$ The evolution of this density depends on three types of mechanisms: displacements, births and deaths. The reaction term represents the rate of births. If we suppose that the individuals can be destroyed by some predators during their displacements, then the dissipative gradient term represents the density of predators.\\

In this paper, we are concerned with solution $u$ of \eqref{eq11}-\eqref{eq13} which blows up in the $L^{\infty}$ norm in the following sense : there exists $T^{*}<\infty$, called the blow up time such that the solution $u$ exists in $[0,T^{*})$ and $$\lim_{t\rightarrow T^{*}}\left\|u(t)\right\|_{L^{\infty}}=+\infty.$$
Numerous articles have been published concerning the problem of global existence or nonexistence of solutions to nonlinear parabolic equations.
Problem \eqref{eq11}-\eqref{eq13}, has been widely analyzed from a mathematical point of view, on the profile, blow up rates, asymptotic behaviours and self similar solutions (see for example: \cite{quittner}, \cite{snoussi}, \cite{pht}, \cite{phweissler}, \cite{souplettayachi} and \cite{zaag}), but to our knowledge, there are no results concerning its numerical approximation.\\

Let us first have a look at the theoretical analysis of this problem. The quasilinear parabolic equation \eqref{eq11} was introduced in $1989$ by Chipot and Weissler (\cite{chipot}) in order to investigate the effect of a damping term on global existence or nonexistence of solutions. They proved local existence, uniqueness and regularity for the problem and they showed the following theorem
\begin{th1} \cite{chipot}
\label{regularite}
Suppose $s\geq q,\ s>d(q-1),\ s\geq \dfrac{dp}{d+p},\ s>\dfrac{d(p-1)}{p+1},\ s\geq 2q$ and $s>dq.$\\
Let $u_{0}\in D:=\{\phi\in W^{3,s}(\Omega)\cap W_{0}^{1,s}(\Omega)\ \  \text{ such that}\ \ \Delta \phi+\left|\phi\right|^{p-1}\phi-\left|\nabla \phi\right|^{q}\in W_{0}^{1,s}(\Omega)\}$ and let $u(t)$ be the maximal solution of the integral equation associated to \eqref{eq11}-\eqref{eq13}
\begin{equation}
u(t)=e^{t\Delta}u_{0}+\int\limits_{0}^{t}e^{(t-s)\Delta}\left(\left|u\right|^{p-1}u-\left|\nabla u\right|^{q}\right)(s)ds
\label{integrale}
\end{equation}
where $e^{t\Delta}$ denotes the heat semigroup on $\Omega$ with Dirichlet conditions. Then we have
\begin{enumerate}
	\item $u\in C^{1}\left([0,T^{*});W_{0}^{1,s}(\Omega)\right)$ and
    $u'(t)=\Delta u+\left|u\right|^{p-1}u-\left|\nabla u\right|^{q}$ where each term on the right side is in $C\left([0,T^{*});L^{\frac{s}{q}}(\Omega)\right)$.
    \item $u\in C\left((0,T^{*});W^{2,\frac{s}{q}}(\Omega)\right)$.
    \item $\left\|u(t)\right\|_{\infty}$ and $\left\|\nabla u(t)\right\|_{\infty}$ are bounded on any interval $[0,T]$ with $T<T^{*}.$
\end{enumerate}
\label{regularitechipot}
\end{th1}
\begin{rem}
	$D$ is the domain of the generator $B$ of the semi-flow $W_{t}$ on $W_{0}^{1,s}(\Omega)$ resulting from the integral equation \eqref{integrale} (for more details see \cite{chipot}).
\end{rem}
They proved that under appropriate conditions on $q,\ p$ and $d$, there exists a suitable initial value $u_{0}$ so that the corresponding solution of \eqref{eq11}-\eqref{eq13} blows up in a finite time. More precisely
\begin{th1} \cite{chipot}
	\label{explosion}
	Let $p>1,\ \ 1\leq q\leq \frac{2p}{p+1}$ and $u_{0}\in W^{3,s}(\Omega)$ for $s$ satisfying the same conditions as in theorem \ref{regularitechipot}, $u_{0}$ not identically zero. In addition, we suppose that:
	\begin{enumerate}
		\item $u_{0}=0$ on $\partial\Omega.$
		\item $\Delta u_{0}-|\nabla u_{0}|^{q}+|u_{0}|^{p}=0$ on $\partial\Omega.$
		\item $u_{0}\geq0$ in $\Omega.$
		\item $\Delta u_{0}-|\nabla u_{0}|^{q}+u_{0}^{p}\geq0$ in $\Omega.$
		\item $E(u_{0})=\frac{1}{2}\left\|\nabla u_{0}\right\|_{2}^{2}-\frac{1}{p+1}\left\|u_{0}\right\|_{p+1}^{p+1}\leq0.$
		\item If $q<\frac{2p}{p+1}$ then $\left\|u_{0}\right\|_{p+1}$ is sufficiently large.
		\item If $q=\frac{2p}{p+1}$ then $p$ is sufficiently large.
	\end{enumerate}
	Then, the corresponding solution of \eqref{eq11}-\eqref{eq13} blows up in finite time in the $L^{\infty}$ norm.
\end{th1}
The obvious difficulty with this result is that it is not at all clear if such a $u_{0}$ exists.\\

Souplet and Weissler have proved in \cite{phweissler} that in a possibly unbounded regular domain $\Omega$, finite time blow up occurs in $W^{1,s}_{0}$ norm ($s$ sufficiently large), for large data whenever $p>q$, and they give an estimation for the blow up time $T^{*}$. More precisely, we have :
\begin{th1} \cite{phweissler}
	\label{tempsexplosion}
	Assume $p>q$ and let $\psi \in W^{1,s}_{0}(\Omega)$ ($s$ sufficiently large) with $\psi \geq 0,\ (\psi \not\equiv 0)$.
	\begin{enumerate}
		\item There exists some $\lambda_{0}=\lambda_{0}(\psi)>0$ such that for all $ \lambda>\lambda_{0}$, the solution of \eqref{eq11}-\eqref{eq13} with initial data $\phi=\lambda \psi$ blows up in finite time in $W^{1,s}_{0}$ norm. If either $\Omega$ is bounded or $q<2$, then blow up occurs in $L^{\infty}$ norm.
		\item If $\Omega$ is bounded, there is some $C(\psi)>0$ such that 
		\begin{equation*}
	 \dfrac{1}{(p-1)(\lambda \left| \psi \right|_{\infty})^{p-1}}	\leq T^{*}(\lambda \psi)\leq \dfrac{C(\psi)}{(\lambda \left| \psi \right|_{\infty})^{p-1}},\ \ \ \lambda\rightarrow \infty.
		\end{equation*}
	\end{enumerate}
\end{th1}
\noindent For more details about the result of regularity of the solution and the conditions of blow up, see \cite{chipot, hessaaraki, pavol, poincare, phweissler} and the references therein.\\ 

In \cite{filacheb, ph, souplettayachi}, authors have proved the following estimations about the blow-up rate:
\begin{th1}
	Assume $q<\dfrac{2p}{p+1}, \ p\leq 1+\dfrac{2}{d+1}$ and let $u\geq 0$ be a solution of \eqref{eq11}-\eqref{eq13} such that $T^{*}<+\infty$. Then, we have
	\begin{equation*}
	C_{1}\left(T^{*}-t\right)^{-\frac{1}{p-1}}\leq \left\|u(t)\right\|_{\infty}\leq 	C_{2}\left(T^{*}-t\right)^{-\frac{1}{p-1}} \ \ \ \text{ as } t\longrightarrow T^{*}.
	\end{equation*}
	\label{tauxexplosion}
\end{th1}
In \cite{hessaaraki}, Hesaaraki and Moameni have proved that for small values of initial data, solutions of \eqref{eq11}-\eqref{eq13} can not blow up in finite time whenever $\Omega$ is  a bounded domain in $\mathbb{R}^{d}.$
 \begin{th1}
 	Let $\Omega$ be a bounded domain. If $u_{0}\geq 0$ and if $u$ is the solution of \eqref{eq11}-\eqref{eq13}, then there exists an $\epsilon>0$ such that, if $\left|u_{0}\right|_{C^{1}(\bar{\Omega})}<\epsilon $ then $u$ is a global solution and decays exponentially.
 	\label{petitedonnee}
 \end{th1}
The phenomenon of blow up in finite time for nonlinear parabolic equations has been extensively studied for last decades. Several papers contain numerous references on blow up results, see for example \cite{bandle}, \cite{friedman}, \cite{fujita}, \cite{hayakawa} and \cite{levine}. There has been many works concerning numerical computation of solutions of nonlinear parabolic equation (see \cite{Chen}, \cite{cho}, \cite{hamada} and \cite{nakagawa}) but without the gradient term. By studying various papers, we found that many interesting numerical problems for the Chipot-Weissler equation are left unsolved, and we would like to solve some of them in this and other forthcoming works. The results of this paper are used to study the properties of the numerical solution associated  with \eqref{eq11}-\eqref{eq13} and hence we can assimilate the dissipative role of the gradient term (\cite{hani}).\\

Although the details are explained in the subsequent sections, we outline here the main ideas of this study. Let us recall a result of Chen (\cite{Chen}) where he considered the Fujita equation $u_{t}=u_{xx}+u^{1+\alpha}$ and proved that some numerical solutions can blow up in finite time at more than one point. The following questions may naturally arise:
\begin{enumerate}
	\item What happens numerically, if we add a gradient term in the Fujita equation?
	\item Which conditions on the reaction term and the gradient term provide or prevent blow-up?
	\item Can we compare the blow-up rates of the equations with and without gradient term?
\end{enumerate}
In this paper, we develop a numerical scheme in order to approximate the solutions of the nonlinear Chipot-Weissler equation in $\Omega=(-1,1)$, and we show that the finite difference solution blows up in finite time if a certain condition is assumed. Next in \cite{hani}, we study the numerical blow up set and the asymptotic behaviours of our numerical solution near the blow up point, we prove the local convergence of the numerical solution to the exact one, we also try to give an approximation of the blow up time.\\

Our paper is organized as follows: In the next section, we present some properties of the exact solution. In section 3, we construct a finite difference scheme and we prove that if the initial data is positive, monotone and symmetric then the numerical solution is also positive, monotone and symmetric. In section 4, we shall prove that the difference solution blows up in $x=0$ the middle of the interval $(-1,1)$. In section 5, we give some numerical results to illustrate our analysis. In the last section, we present some interesting questions which will be solved in the future study.
\section{Properties of the exact solution}
We consider the semilinear equation
\begin{equation}
u_{t}=u_{xx}+|u|^{p-1}u-|u_{x}|^{q}  \text {  in   } ]0,T[\times]-1,1[,
\label{e21}
\end{equation}
with initial data
\begin{equation}
u(0,x)=u_{0}(x)  \text {  for   } x\in [-1,1]
\label{22}
\end{equation}
and Dirichlet boundary conditions
\begin{equation}
u(t,-1)=u(t,1)=0 \text { for  } t\in [0,T].  
\label{23}
\end{equation}
where $p>1$, $1\leq q\leq \frac{2p}{p+1}$ and $T<T^{*}.$\\
Here if $t\in [0,T]$ then $\left\|u(t,.)\right\|_{\infty}:=\max\limits_{x}|u(t,x)|<\infty.$\\
For the sake of simplicity, we assume that the initial data $u_{0}$ satisfies the following conditions:\\
\textbf{(A1)} $u_{0}\in C^{1}((-1,1))$, nonconstant and nonnegative in $[-1,1].$\\
\textbf{(A2)} $u_{0}$ is spatially symmetric about $x=0.$\\
\textbf{(A3)} $u_{0}$ is strictly monotone increasing in $[-1,0].$\\
\textbf{(A4)} $u_{0}$ is large in the sense that $\left\|u_{0}\right\|_{\infty}>>1.$\\
\textbf{(A5)} $u_{0}(-1)=u_{0}(1)=0.$\\
These properties will be preserved by our numerical scheme and make computations easier.\\
Under these conditions, it is known from \cite{Cheblikfila} that the solution blows up only at the central point, that is
\begin{equation*}
\exists \ \ T^{*}<+\infty \text{ such that } \lim_{t\rightarrow T^{*}}u(t,0)=+\infty \text{ but } \lim_{t\rightarrow T^{*}}u(t,x)<\infty \text{  when } x\neq 0.
\end{equation*}
In this section, we study some properties of $u$ the classical solution of \eqref{e21}-\eqref{23} with initial data $u_{0}$ satisfying \textbf{(A1)-(A5)}. 
\subsection{Positivity}
Chipot and Weissler have proved in \cite{chipot} that
\begin{lem}
	\label{positiviteexacte}
	For $s$ sufficiently large such that $s\geq 2q$, if $u_{0}\in W^{1,s}_{0}((-1,1))$ with $u_{0}\geq 0$ in $[-1,1]$, then $u(t)\geq 0$ for all $t\in [0,T^{*}).$
\end{lem}
\subsection{Symmetry}
\begin{lem}
	\label{symetrieexacte}
	We suppose that the initial data $u_{0}\in W_{0}^{1,s}((-1,1))$ and satisfies \textbf{(A1)-(A5)}, then the exact solution $u$ of \eqref{e21}-\eqref{23} is symmetric, that is:
	\begin{equation*}
	\forall x\in [-1,1],\ \ u(t,x)=u(t,-x) \ \ \forall t\in [0,T^{*}).
	\end{equation*}
\end{lem}
\begin{proof}
	Let first assume that $u_{0}\in D$, by theorem \ref{regularite} and using continuous dependence of the solutions to the initial data in $ W_{0}^{1,s}((-1,1))$ (\cite{poincare}, Prop. A1) and the embedding $ W_{0}^{1,s}(\Omega)\subset C(\bar{\Omega})$, we can show that the result of lemma \ref{symetrieexacte} is true for all $u_{0}\in  W_{0}^{1,s}((-1,1))$.\\
	Let $u(t,x)$ be the solution of \eqref{e21}-\eqref{23}. We define, for all $t\geq 0$, the function
	$$
	v(t,x)=
	\left\{
	\begin{array}{lll}
	u(t,-x) \ \  x\in [-1,0]\\
	u(t,x) \ \  x\in [0,1].\\
	\end{array}
	\right.
	$$
	We shall prove that $v$ is a solution of \eqref{e21}-\eqref{23} in $[-1,1].$
	\begin{itemize}
		\item In $[0,1],$ we have $v(t,x)=u(t,x)$. Then for all $x\in ]0,1[ \text{ and } t\geq 0,\ v$ satisfies,
		\begin{equation}
		\left\{
		\begin{array}{lll}
		\dfrac{\partial v}{\partial t}(t,x)-\dfrac{\partial^{2} v}{\partial x^{2}}(t,x)-\left|v(t,x)\right|^{p}+\left|\dfrac{\partial v}{\partial x}(t,x)\right|^{q}=0\\
		v(t,1)=0.\\
		\end{array}
		\right.
		\label{sys1}
		\end{equation}	  
		Then $v$ is a solution of \eqref{e21} in $[0,1[.$
		\item In $]-1,0],$ we have $v(t,x)=u(t,-x).$ Then for $x\in ]-1,0]\ \ and\ \ t\geq 0$ we get
		\begin{eqnarray*}
			\nonumber &&\dfrac{\partial v}{\partial t}(t,x)-\frac{\partial^{2} v}{\partial x^{2}}(t,x)-\left|v(t,x)\right|^{p}+\left|\dfrac{\partial v}{\partial x}(t,x)\right|^{q}\\
			&=&\dfrac{\partial u}{\partial t}(t,-x)-\dfrac{\partial^{2} u}{\partial x^{2}}(t,-x)-\left|u(t,-x)\right|^{p}+\left|\dfrac{\partial u}{\partial x}(t,-x)\right|^{q}.\\
		\end{eqnarray*}
		Since $-x\in[0,1[$ and $u(t,-x)$ is a solution in $[0,1[$, we deduce then that
		\begin{equation}
		\left\{
		\begin{array}{lll}
		\dfrac{\partial v}{\partial t}(t,x)-\dfrac{\partial^{2} v}{\partial x^{2}}(t,x)-\left|v(t,x)\right|^{p}+\left|\dfrac{\partial v}{\partial x}(t,x)\right|^{q}=0\\
		v(t,-1)=0.\\
		\end{array}
		\right.
		\label{sys2}
		\end{equation}
		Using \textbf{(A2)}, we obtain
		\begin{equation}
		v(0,x)=u_{0}\ \ \forall x\in [-1,1].
		\label{sys3}
		\end{equation}
		Finally, by \eqref{sys1}, \eqref{sys2}, \eqref{sys3} and by unicity of the solution of \eqref{e21}-\eqref{23}, we get $u=v \ \ in \ \ [-1,1].$\\ 
		This finishes the proof of the symmetry.
	\end{itemize}
\end{proof}
\subsection{Monotony}
\begin{lem}
	\label{monotonieexacte}
	We suppose that the initial data $u_{0}\in W_{0}^{1,s}((-1,1))$ and satisfies \textbf{(A1)-(A5)}, then the exact solution $u$ of \eqref{e21}-\eqref{23} is increasing in $[-1,0[$, that is:
	\begin{equation*}
	\forall x \in [-1,0[\ \text{ and }  \ \ t\geq 0 \ \text{ we have }\ \  \dfrac{\partial u}{\partial x}(t,x)\geq 0.
	\end{equation*}
\end{lem}
\begin{proof}
	Similar to the proof of lemma \ref{symetrieexacte}, here we assume that $u_{0}\in D$. Let $u(t,x)$ be the solution of \eqref{e21}-\eqref{23} in $[-1,0[ $, and $v(t,x)=\dfrac{\partial u}{\partial x}(t,x).$\\
	For all $x\in [-1,0[$ and $t>0$, we have
	\begin{equation*}
	\dfrac{\partial u}{\partial t}(t,x)-\dfrac{\partial^{2} u}{\partial x^{2}}(t,x)-u^{p}(t,x)+\left|\dfrac{\partial u}{\partial x}(t,x)\right|^{q}=0.
	\end{equation*}
	Then
	\begin{equation}
	\dfrac{\partial v}{\partial t}(t,x)-\dfrac{\partial^{2} v}{\partial x^{2}}(t,x)-pu^{p-1}(t,x)v(t,x)+q\dfrac{\partial v}{\partial x}sign(v)\left|v\right|^{q-1}(t,x)=0.
	\label{mono}
	\end{equation}
	Let 
	\begin{equation*}
	v^{-}=\max(0,-v).
	\end{equation*}
	Multiplying \eqref{mono} by $v^{-}(t,x)$, we get
	\begin{eqnarray*}
		&&\dfrac{\partial v}{\partial t}v^{-}-\dfrac{\partial^{2} v}{\partial x^{2}}v^{-}-pu^{p-1}vv^{-}-q\dfrac{\partial v}{\partial x}v^{-}\left|v\right|^{q-1}=0\\
		&\Rightarrow& -\dfrac{\partial v^{-}}{\partial t}v^{-}+\dfrac{\partial^{2} v^{-}}{\partial x^{2}}v^{-}+pu^{p-1}(v^{-})^{2}+q\dfrac{\partial v^{-}}{\partial x}v^{-}\left|v^{-}\right|^{q-1}=0.
	\end{eqnarray*}
	Integrating over $[-1,1]$, we get
	\begin{equation*}
		\int_{-1}^{1}{\dfrac{\partial v^{-}}{\partial t}v^{-}dx}-\int_{-1}^{1}{\dfrac{\partial^{2} v^{-}}{\partial x^{2}}v^{-}dx}-p\int_{-1}^{1}{u^{p-1}(v^{-})^{2}dx}-q\int_{-1}^{1}{\dfrac{\partial v^{-}}{\partial x}(v^{-})^{q}dx}=0
	\end{equation*}
	then
	\begin{equation*} \dfrac{1}{2}\dfrac{d}{dt}\int_{-1}^{1}{(v^{-})^{2}dx}-\left(\left[\dfrac{\partial v^{-}}{\partial x}v^{-}\right]_{-1}^{1}-\int_{-1}^{1}{\left(\dfrac{\partial v^{-}}{\partial x}\right)^{2}dx}\right)-p\int_{-1}^{1}{u^{p-1}(v^{-})^{2}dx}
	-q\int_{-1}^{1}{\dfrac{\partial v^{-}}{\partial x}(v^{-})^{q}dx}=0.
	\end{equation*}
	But
	\begin{eqnarray}
	\nonumber \left[\dfrac{\partial v^{-}}{\partial x}v^{-}\right]_{-1}^{1}&=&\dfrac{\partial v^{-}}{\partial x}(t,1)v^{-}(t,1)-\dfrac{\partial v^{-}}{\partial x}(t,-1)v^{-}(t,-1)\\
	\nonumber &=&2v^{-}(t,1)\dfrac{\partial v^{-}}{\partial x}(t,1)\\
	&=&-2v^{-}(t,1)\dfrac{\partial v}{\partial x}(t,1).
	\label{mon}
	\end{eqnarray}
	On the other hand, from \eqref{23}, we have 
	\begin{equation*}
	u^{p}(t,1)=\dfrac{\partial u}{\partial t}(t,1)=0,
	\end{equation*}
	then by \eqref{e21} 
	\begin{equation*}
	\dfrac{\partial v}{\partial x}(t,1)=\dfrac{\partial^{2} u}{\partial x^{2}}(t,1)=\left|\dfrac{\partial u}{\partial x}(t,1)\right|^{q}\geq 0,
	\end{equation*}
	and then we can deduce from \eqref{mon} that
	\begin{equation*}
	\left[\dfrac{\partial v^{-}}{\partial x}v^{-}\right]_{-1}^{1}\leq 0.
	\end{equation*}
	So
	\begin{eqnarray*}
		&&\dfrac{1}{2}\dfrac{d}{dt}\int_{-1}^{1}{(v^{-})^{2}dx}+\int_{-1}^{1}{\left(\dfrac{\partial v^{-}}{\partial x}\right)^{2}dx}-p\int_{-1}^{1}{u^{p-1}(v^{-})^{2}dx}-q\int_{-1}^{1}{\dfrac{\partial v^{-}}{\partial x}(v^{-})^{q}dx}\\
		&=&\left[\dfrac{\partial v^{-}}{\partial x}v^{-}\right]_{-1}^{1}\\
		&\leq& 0.
	\end{eqnarray*}
	Using the symmetry property, we get
	\begin{eqnarray*}
		&&\dfrac{d}{dt}\int_{-1}^{0}{(v^{-})^{2}dx}+2\int_{-1}^{0}{\left(\dfrac{\partial v^{-}}{\partial x}\right)^{2}dx}-2p\int_{-1}^{0}{u^{p-1}(v^{-})^{2}dx}-2q\int_{-1}^{0}{\dfrac{\partial v^{-}}{\partial x}(v^{-})^{q}dx}\leq 0\\
		&&\Rightarrow \dfrac{1}{2}\dfrac{d}{dt}\int_{-1}^{0}{(v^{-})^{2}dx}\leq -\int_{-1}^{0}{\left(\dfrac{\partial v^{-}}{\partial x}\right)^{2}dx}+p\int_{-1}^{0}{u^{p-1}(v^{-})^{2}dx}+q\left|\int_{-1}^{0}{\dfrac{\partial v^{-}}{\partial x}(v^{-})^{q}dx}\right|.
	\end{eqnarray*}
	We refer to theorem \ref{regularite} proved in \cite{chipot}, we can see that $u$ and $\nabla u$ are bounded before blow up, then there exists $M,\ N>0$ such that
	\begin{eqnarray*}
		&&\left|u^{p-1}(t,x)\right|\leq M \text{ for all } x\in[-1,1] \text{ and } 0<t<T,\\
		&&\left|\nabla u(t,x)\right|\leq N \text{ for all } x\in[-1,1] \text{ and } 0<t<T.
	\end{eqnarray*} 
	We use Young's inequality, then for all $\epsilon>0$, there exists $C_{\epsilon}>0$, such that
	\begin{equation*}
	\left|\int_{-1}^{0}{\dfrac{\partial v^{-}}{\partial x}(v^{-})^{q}dx}\right|\leq \epsilon \int_{-1}^{0}{\left(\dfrac{\partial v^{-}}{\partial x}\right)^{2}dx}+C_{\epsilon}\int_{-1}^{0}{(v^{-})^{2q}dx}.
	\end{equation*}
	Which implies that
	\begin{equation*}
	\left|\int_{-1}^{0}{\dfrac{\partial v^{-}}{\partial x}(v^{-})^{q}dx}\right|\leq \epsilon \int_{-1}^{0}{\left(\dfrac{\partial v^{-}}{\partial x}\right)^{2}dx}+C_{\epsilon}N^{2q-2}\int_{-1}^{0}{(v^{-})^{2}dx}.
	\end{equation*}
	Then we get
	\begin{eqnarray*}
		\dfrac{1}{2}\dfrac{d}{dt}\int_{-1}^{0}{(v^{-})^{2}dx}&\leq& -\int_{-1}^{0}{\left(\dfrac{\partial v^{-}}{\partial x}\right)^{2}dx}+M\int_{-1}^{0}{(v^{-})^{2}dx}+\epsilon q \int_{-1}^{0}{\left(\dfrac{\partial v^{-}}{\partial x}\right)^{2}dx}\\
		&& \ \ \ \ \ +C_{\epsilon}q N^{2q-2}\int_{-1}^{0}{(v^{-})^{2}dx}\\
		&=&(\epsilon q-1)\int_{-1}^{0}{\left(\dfrac{\partial v^{-}}{\partial x}\right)^{2}dx}+M_{1}\int_{-1}^{0}{(v^{-})^{2}dx},
	\end{eqnarray*}
	where $M_{1}$ is a constant depending on $N,\ M,\ q$ and $C_{\epsilon}.$\\
	For $\epsilon$ sufficiently small we get
	\begin{equation*}
	\dfrac{d}{dt}\int_{-1}^{0}{(v^{-})^{2}dx}\leq 2 M_{1}\int_{-1}^{0}{(v^{-})^{2}dx}.
	\end{equation*}
	Integrating over $[0,t]$, for $t \leq T$, we get
	\begin{equation*}
	\int_{-1}^{0}{(v^{-})^{2}(t,x)dx}\leq \int_{-1}^{0}{(v^{-})^{2}(0,x)dx}+M\int_{0}^{t}{\int_{-1}^{0}{(v^{-})^{2}(s,x)dx}ds}.
	\end{equation*}
	From \textbf{(A2)} and \textbf{(A3)}, we have
	\begin{equation*}
	v^{-}(0,x)=\left(\dfrac{\partial u}{\partial x}\right)^{-}(0,x)=\left(\dfrac{\partial u_{0}}{\partial x}\right)^{-}(x)=0.
	\end{equation*}
	Then using Gronwall lemma, we deduce that $\forall x\in [-1,0[,\ t\geq 0,\ \ v^{-}(t,x)=0$, which achieves the proof of monotony.
\end{proof}
\section{Full discretization}
 We consider the semilinear parabolic equation for $d=1$ and $\Omega=(-1,1)$
\begin{equation}
\left\{ 
\begin{array}{llll}
u_{t}=u_{xx}+|u|^{p-1}u-|u_{x}|^{q} \  \text {  in   } \ (0,T)\times(-1,1),\\
u(0,x)=u_{0}(x) \  \text {  for   }\  x\in (-1,1),\\
u(t,-1)=u(t,1)=0 \  \text { for  }\  t\in (0,T).  \\
\end{array}
\right.
\label{exacte}
\end{equation}
In this section, we construct a finite difference scheme which solution satisfies the same properties proved above.\\
Throughout this paper, we use the following notations in the list below:
\begin{enumerate}
	\item $\tau: $ size parameter for the variable time mesh $\tau_{n}$.
	\item $h: $ positive parameter for which $\lambda:=\dfrac{\tau}{h^{2}}<\dfrac{1}{16}$ kept fixed.	 
	\item $t_{n}:\ \ n$-th time step determined as:
	$$
	\left\{ 
	\begin{array}{lll}
	t_{0}=0\\
	t_{n}=t_{n-1}+\tau_{n-1}=\displaystyle\sum\limits_{k=0}^{n-1}{\tau_{k}}, \ n\geq 1.\\
	\end{array}
	\right.\\
	$$ 	 
	\item $x_{j}:\  j$-th net point on $[-1,1]$ determined as:
	$$
	\left\{ 
	\begin{array}{lll}
	x_{0}=-1\\
	x_{j}=x_{j-1}+h_{n},\  j\geq 1 \  and \  n\geq 0\\
	x_{N_{n}+1}=1.
	\end{array}
	\right.\\
	$$
	\item $u_{j}^{n}: $ approximate of $u(t_{n},x_{j}).$
	\item $\tau_{n}: $ discrete time increment of $(n+1)-$th step determined by $\tau_{n}=\tau \min\left(1,\left\|u^{n}\right\|_{\infty}^{-p+1}\right)$ for $n\geq 0.$ 
	\item $h_{n}: $ discrete space increment of $(n+1)-$th step determined by $h_{n}=\min\left(h,\left(2\left\|u^{n}\right\|_{\infty}^{-q+1}\right)^{\frac{1}{2-q}}\right)$ for $n\geq 0.$ 
	\item $N_{n}=E\left(\dfrac{1}{h_{n}}-1\right)+1$ where $E(X)$ is the integer part of $X$.
\end{enumerate}
Under the assumption that a spatial net point $x_{m}$ coincides with the middle point $x=0 $ (we can easily achieve this by taking $N_{n}+1=2m$), we will prove that
\begin{equation*}
u_{m}^{n}=\max_{0\leq j\leq N_{n}+1}|u_{j}^{n}|=\left\|u^{n}\right\|_{\infty} \text{ and } u_{m-1}^{n}=\max_{j\neq m}|u_{j}^{n}|.
\end{equation*} 
By using the notations above, our difference equation is introduced by: for $j=1,...,N_{n}$ and $n\geq 0$:
	\begin{equation}
	\left\{ 
	\begin{array}{lllll}
	\dfrac{u_{j}^{n+1}-u_{j}^{n}}{\tau_{n}}=\dfrac{u_{j+1}^{n+1}-2u_{j}^{n+1}+u_{j-1}^{n+1}}{h_{n}^{2}}+(u_{j}^{n})^{p}-\dfrac{1}{(2h_{n})^{q}}\left|u_{j+1}^{n}-u_{j-1}^{n}\right|^{q-1}\left|u_{j+1}^{n+1}-u_{j-1}^{n+1}\right|\\
	u_{j}^{0}=u_{0}(x_{j})\\
	u_{0}^{n}=u_{N_{n}+1}^{n}=0.\\
	\end{array}
	\right.
	\label{approchee}
	\end{equation}
We denote by $U^{n}:=(u_{0}^{n},...,u_{N_{n}+1}^{n})^{t}.$\\
\begin{rem}
	\begin{enumerate}
	\item Under the assumptions \textbf{(A1)-(A5)} of the initial data, problem \eqref{approchee} has a unique solution and the proof is given in the proof of lemma \ref{symetrienum}.
	\item In this work, we study some properties of the exact solution of \eqref{exacte}, the numerical solution of \eqref{approchee} and we show that we have the same properties, this is independently of the convergence of the scheme. The convergence of solution of \eqref{approchee} to the solution of \eqref{exacte} is proved in \cite{hani} and the proof is based on these properties and other properties of the numerical solution discussed in \cite{hani}. 
	\item  In \cite{hani}, we have proved the local convergence of the numerical solution to the exact one far from the blow-up set. 
	\end{enumerate}
\end{rem}
Next, we prove that the difference solution has the same properties as the exact one.
\subsection{Positivity:}
We will prove now the analogue of lemma \ref{positiviteexacte} for the difference solution.
\begin{lem}
	Suppose $u_{0}$ is a positive function in $[-1,1]$. Let $U^{n}$ be the solution of \eqref{approchee}, then we have $U^{n}\geq 0$ for all $n\geq 0.$
\end{lem}
\begin{proof}
	In view of the assumption that the initial data satisfies $u_{0}\geq 0$ and using the notation $U^{n}:=(u_{0}^{n},...,u_{N_{n}+1}^{n})^{t}$, we see that $U^{n}\geq 0$ holds for $n=0$. Supposing that it holds for some $n\geq 0$, we have to show that $U^{n+1}\geq 0$.\\
	For all $j=1,...,N_{n}$, define
	\begin{equation*}
	(u_{j}^{n+1})^{-}:=\max_{j}(0,-u_{j}^{n+1}) \text{  and  } (u_{j}^{n+1})^{+}:=\max_{j}(0,u_{j}^{n+1}).
	\end{equation*}
	If we multiply the first equation of \eqref{approchee} by $(u_{j}^{n+1})^{-}$ we obtain
	\begin{eqnarray*}
		\dfrac{u_{j}^{n+1}-u_{j}^{n}}{\tau_{n}}(u_{j}^{n+1})^{-} =&&\dfrac{u_{j+1}^{n+1}-2u_{j}^{n+1}+u_{j-1}^{n+1}}{h_{n}^{2}}(u_{j}^{n+1})^{-}+(u_{j}^{n})^{p}(u_{j}^{n+1})^{-} \\
		&&-\dfrac{1}{(2h_{n})^{q}}\left|u_{j+1}^{n}-u_{j-1}^{n}\right|^{q-1}\left|u_{j+1}^{n+1}-u_{j-1}^{n+1}\right|(u_{j}^{n+1})^{-}.
	\end{eqnarray*}
	We use
	\begin{align*}
	u_{j}^{n+1}=(u_{j}^{n+1})^{+}-(u_{j}^{n+1})^{-} \ \ and \ \ (u_{j}^{n+1})^{+}.(u_{j}^{n+1})^{-}=0.
	\end{align*}
	Then we have
	{\footnotesize
		\begin{eqnarray*}
			& -&\dfrac{(u_{j}^{n+1})^{-}-(u_{j}^{n})^{-}}{\tau_{n}}(u_{j}^{n+1})^{-}-\dfrac{1}{h_{n}^{2}}(u_{j+1}^{n+1})^{+}(u_{j}^{n+1})^{-}+\dfrac{1}{h_{n}}\frac{(u_{j+1}^{n+1})^{-}-(u_{j}^{n+1})^{-}}{h_{n}}(u_{j}^{n+1})^{-}\\
			&&-\dfrac{1}{h_{n}^{2}}(u_{j-1}^{n+1})^{+}(u_{j}^{n+1})^{-}-\dfrac{1}{h_{n}}\dfrac{(u_{j}^{n+1})^{-}-(u_{j-1}^{n+1})^{-}}{h_{n}}(u_{j}^{n+1})^{-}
			+\dfrac{1}{2}\left|\dfrac{u_{j+1}^{n}-u_{j-1}^{n}}{2h_{n}}\right|^{q-1}\left|\dfrac{u_{j+1}^{n+1}-u_{j-1}^{n+1}}{h_{n}}\right|(u_{j}^{n+1})^{-}\\
			&=& \dfrac{1}{\tau_{n}}(u_{j}^{n})^{+}(u_{j}^{n+1})^{-}+(u_{j}^{n})^{p}(u_{j}^{n+1})^{-} \\
			&\geq & 0
		\end{eqnarray*}
	}
	We multiply by $(-1)$ and we sum for $j=1,...,N_{n}$, we obtain
	{\footnotesize
		\begin{eqnarray*}
			&&\sum_{j=1}^{N_{n}}{\dfrac{(u_{j}^{n+1})^{-}-(u_{j}^{n})^{-}}{\tau_{n}}(u_{j}^{n+1})^{-}}
			+\dfrac{1}{h_{n}^{2}}\sum_{j=1}^{N_{n}}{(u_{j+1}^{n+1})^{+}(u_{j}^{n+1})^{-}}-\dfrac{1}{h_{n}}\sum_{j=1}^{N_{n}}{{\dfrac{(u_{j+1}^{n+1})^{-}-(u_{j}^{n+1})^{-}}{h_{n}}(u_{j}^{n+1})^{-}}} \\
			&&+\dfrac{1}{h_{n}^{2}}\sum_{j=1}^{N_{n}}{(u_{j-1}^{n+1})^{+}(u_{j}^{n+1})^{-}}
			+\dfrac{1}{h_{n}}\sum_{j=1}^{N_{n}}{\dfrac{(u_{j}^{n+1})^{-}-(u_{j-1}^{n+1})^{-}}{h_{n}}(u_{j}^{n+1})^{-}} \\ 
			& \leq & \dfrac{1}{2}\sum_{j=1}^{N_{n}}\dfrac{1}{(2h)^{q}}\left|u_{j+1}^{n}-u_{j-1}^{n}\right|^{q-1}\left|u_{j+1}^{n+1}-u_{j-1}^{n+1}\right|(u_{j}^{n+1})^{-}.
		\end{eqnarray*}
	}
	But 
	\begin{equation*}
	\sum_{j=1}^{N_{n}}{(u_{j-1}^{n+1})^{+}(u_{j}^{n+1})^{-}}=\sum_{j=1}^{N_{n}}{(u_{j}^{n+1})^{+}(u_{j+1}^{n+1})^{-}},
	\end{equation*}
	and 
	\begin{equation*}
	\sum_{j=1}^{N_{n}}{{\dfrac{(u_{j}^{n+1})^{-}-(u_{j-1}^{n+1})^{-}}{h_{n}}(u_{j}^{n+1})^{-}}}=\sum_{j=1}^{N_{n}}{{\dfrac{(u_{j+1}^{n+1})^{-}-(u_{j}^{n+1})^{-}}{h_{n}}(u_{j+1}^{n+1})^{-}}}+\dfrac{((u_{1}^{n+1})^{-})^{2}}{h_{n}}.
	\end{equation*}
	Then
	{\footnotesize
		\begin{eqnarray*}
			&&\sum_{j=1}^{N_{n}}{\dfrac{(u_{j}^{n+1})^{-}-(u_{j}^{n})^{-}}{\tau_{n}}(u_{j}^{n+1})^{-}}
			+\dfrac{1}{h_{n}^{2}}\sum_{j=1}^{N_{n}}{(u_{j+1}^{n+1})^{+}(u_{j}^{n+1})^{-}}-\dfrac{1}{h_{n}}\sum_{j=1}^{N_{n}}{{\dfrac{(u_{j+1}^{n+1})^{-}-(u_{j}^{n+1})^{-}}{h_{n}}(u_{j}^{n+1})^{-}}} \\
			&&+\dfrac{1}{h_{n}^{2}}\sum_{j=1}^{N_{n}}{(u_{j-1}^{n+1})^{+}(u_{j}^{n+1})^{-}}
			+\dfrac{1}{h_{n}}\sum_{j=1}^{N_{n}}{\dfrac{(u_{j}^{n+1})^{-}-(u_{j-1}^{n+1})^{-}}{h_{n}}(u_{j}^{n+1})^{-}} \\ 
			&=&\sum_{j=1}^{N_{n}}{\dfrac{(u_{j}^{n+1})^{-}-(u_{j}^{n})^{-}}{\tau_{n}}(u_{j}^{n+1})^{-}}+\dfrac{1}{h_{n}^{2}}\sum_{j=1}^{N_{n}}{\left((u_{j+1}^{n+1})^{+}(u_{j}^{n+1})^{-}+(u_{j}^{n+1})^{+}(u_{j+1}^{n+1})^{-}\right)}\\
			&&+\sum_{j=1}^{N_{n}}\left({{\dfrac{(u_{j+1}^{n+1})^{-}-(u_{j}^{n+1})^{-}}{h_{n}}}}\right)^{2}+\left(\dfrac{(u_{1}^{n+1})^{-}}{h_{n}}\right)^{2}.
		\end{eqnarray*}
	}
	We use that $M_{n}:=\left\|U^{n}\right\|_{\infty}<+\infty$ before blow-up, we can write that 
	\begin{equation*}
	\dfrac{|u_{j+1}^{n}-u_{j-1}^{n}|}{2h_{n}}\leq \dfrac{M_{n}}{h_{n}},
	\end{equation*}
	which implies that
	{\footnotesize
		\begin{eqnarray*}
			&& \sum_{j=1}^{N_{n}}{\dfrac{(u_{j}^{n+1})^{-}-(u_{j}^{n})^{-}}{\tau_{n}}(u_{j}^{n+1})^{-}}+\dfrac{1}{h_{n}^{2}}\sum_{j=1}^{N_{n}}{\left((u_{j+1}^{n+1})^{+}(u_{j}^{n+1})^{-}+(u_{j}^{n+1})^{+}(u_{j+1}^{n+1})^{-}\right)}\\
			&& +\sum_{j=1}^{N_{n}}\left({{\dfrac{(u_{j+1}^{n+1})^{-}-(u_{j}^{n+1})^{-}}{h_{n}}}}\right)^{2}+\left(\dfrac{(u_{1}^{n+1})^{-}}{h_{n}}\right)^{2}\\
			& \leq&  \dfrac{1}{2}\sum_{j=1}^{N_{n}}\dfrac{1}{(2h)^{q}}\left|u_{j+1}^{n}-u_{j-1}^{n}\right|^{q-1}|u_{j+1}^{n+1}-u_{j-1}^{n+1}|(u_{j}^{n+1})^{-}\\
			& \leq& \dfrac{1}{2}\left(\dfrac{M_{n}}{h_{n}}\right)^{q-1}\dfrac{1}{h_{n}}\sum_{j=1}^{N_{n}}{\left((u_{j+1}^{n+1})^{+}(u_{j}^{n+1})^{-}+(u_{j}^{n+1})^{+}(u_{j+1}^{n+1})^{-}\right)}
			+\dfrac{1}{2}\left(\dfrac{M_{n}}{h_{n}}\right)^{q-1}\dfrac{((u_{1}^{n+1})^{-})^{2}}{h_{n}}\\
			&& +\dfrac{1}{2}\left(\dfrac{M_{n}}{h_{n}}\right)^{q-1}\sum_{j=1}^{N_{n}}{\dfrac{\left|(u_{j+1}^{n+1})^{-}-(u_{j}^{n+1})^{-}\right|}{h_{n}}((u_{j+1}^{n+1})^{-}+(u_{j}^{n+1})^{-})}.
		\end{eqnarray*}
	}
	We define now the operator
	\begin{equation*}
	(D(U^{n+1})^{-})_{j}:=\dfrac{(u_{j+1}^{n+1})^{-}-(u_{j}^{n+1})^{-}}{h_{n}} \text{ for }   j=1,...,N_{n}
	\end{equation*}
	and we denote by $\left\|X\right\|=\sum\limits_{j=1}^{N_{n}}{(X_{j})^{2}}$, then we have
	\begin{eqnarray*}
		&&\left(\dfrac{1}{h_{n}^{2}}-\dfrac{1}{2}\left(\dfrac{M_{n}}{h_{n}}\right)^{q-1}\frac{1}{h_{n}}\right)\sum_{j=1}^{N_{n}}{\left((u_{j+1}^{n+1})^{+}(u_{j}^{n+1})^{-}+(u_{j}^{n+1})^{+}(u_{j+1}^{n+1})^{-}\right)}\\
		&& \ \ \ +\sum_{j=1}^{N_{n}}{\dfrac{(u_{j}^{n+1})^{-}-(u_{j}^{n})^{-}}{\tau_{n}}(u_{j}^{n+1})^{-}}+\left\|D(U^{n+1})^{-}\right\|^{2}\\
		&&\ \ \ \ +\left(\dfrac{1}{h_{n}^{2}}-\dfrac{1}{2}\left(\dfrac{M_{n}}{h_{n}}\right)^{q-1}\dfrac{1}{h_{n}}\right)\left((u_{1}^{n+1})^{-}\right)^{2}\\
		&\leq& \dfrac{1}{2}\left(\dfrac{M_{n}}{h_{n}}\right)^{q-1}\sum_{j=1}^{N_{n}}{\dfrac{\left|(u_{j+1}^{n+1})^{-}-(u_{j}^{n+1})^{-}\right|}{h_{n}}\left((u_{j+1}^{n+1})^{-}+(u_{j}^{n+1})^{-}\right)}.
	\end{eqnarray*}
	Using the definition of $h_{n}$, we get 
	\begin{equation*}
	h_{n}=\min\left(h,(2\left\|u^{n}\right\|_{\infty}^{-q+1})^{\frac{1}{2-q}}\right)\leq \left(2\left\|u^{n}\right\|_{\infty}^{-q+1}\right)^{\frac{1}{2-q}}=\left(2M_{n}^{-q+1}\right)^{\frac{1}{2-q}},
	\end{equation*}
	which implies that 
	\begin{equation}
	\dfrac{1}{h_{n}^{2}}-\dfrac{1}{2}\left(\dfrac{M_{n}}{h_{n}}\right)^{q-1}\dfrac{1}{h_{n}} > 0,
	\label{hypothesepositive}
	\end{equation}
	and we obtain
	\begin{eqnarray*}
		&& \sum_{j=1}^{N_{n}}{\dfrac{(u_{j}^{n+1})^{-}-(u_{j}^{n})^{-}}{\tau_{n}}(u_{j}^{n+1})^{-}}+\left\|D(U^{n+1})^{-}\right\|^{2}\\
		&\leq& \dfrac{1}{2}\left(\dfrac{M_{n}}{h_{n}}\right)^{q-1}\sum_{j=1}^{N_{n}}{\dfrac{\left|(u_{j+1}^{n+1})^{-}-(u_{j}^{n+1})^{-}\right|}{h_{n}}\left((u_{j+1}^{n+1})^{-}+(u_{j}^{n+1})^{-}\right)}.
	\end{eqnarray*}
	For the second term, we use Young's inequality
	\begin{equation*}
	\forall \epsilon>0,\ \ \ ab\leq \epsilon a^{2}+\dfrac{1}{\epsilon} b^{2}
	\end{equation*}
	to obtain that, for all $\epsilon >0$
	\begin{eqnarray*}
		&& \dfrac{1}{2}\left(\dfrac{M_{n}}{h_{n}}\right)^{q-1}\sum_{j=1}^{N_{n}}{\dfrac{\left|(u_{j+1}^{n+1})^{-}-(u_{j}^{n+1})^{-}\right|}{h_{n}}\left((u_{j+1}^{n+1})^{-}+(u_{j}^{n+1})^{-}\right)} \\
		&=& \displaystyle
		\sum_{j=1}^{N_{n}}{\left(\dfrac{\left|(u_{j+1}^{n+1})^{-}-(u_{j}^{n+1})^{-}\right|}{h_{n}}\right)\left(\dfrac{1}{2}\left(\dfrac{M_{n}}{h_{n}}\right)^{q-1}\left((u_{j+1}^{n+1})^{-}+(u_{j}^{n+1})^{-}\right)\right)} \\
		& \leq& \epsilon \sum_{j=1}^{N_{n}}{\left(\dfrac{|(u_{j+1}^{n+1})^{-}-(u_{j}^{n+1})^{-}|}{h_{n}}\right)^{2}}+\dfrac{1}{4\epsilon}\left(\dfrac{M_{n}}{h_{n}}\right)^{2(q-1)}\sum_{j=1}^{N_{n}}{\left((u_{j+1}^{n+1})^{-}+(u_{j}^{n+1})^{-}\right)^{2}}\\
		& \leq& \epsilon \left\|D(U^{n+1})^{-}\right\|^{2}+\dfrac{1}{4\epsilon}\left(\dfrac{M_{n}}{h_{n}}\right)^{2(q-1)}\sum_{j=1}^{N_{n}}{\left((u_{j+1}^{n+1})^{-}+(u_{j}^{n+1})^{-}\right)^{2}} \\ 
		& \leq &\epsilon \left\|D(U^{n+1})^{-}\right\|^{2}+\dfrac{1}{\epsilon}\left(\dfrac{M_{n}}{h_{n}}\right)^{2(q-1)}\left\|(U^{n+1})^{-}\right\|^{2}.
	\end{eqnarray*} 
	If we take $\epsilon =\dfrac{1}{2}$ then we obtain
	\begin{equation*}
	\sum_{j=1}^{N_{n}}{\dfrac{(u_{j}^{n+1})^{-}-(u_{j}^{n})^{-}}{\tau_{n}}(u_{j}^{n+1})^{-}}+\left\|D(U^{n+1})^{-}\right\|^{2}
	\leq \dfrac{1}{2} \left\|D(U^{n+1})^{-}\right\|^{2}+2\left(\dfrac{M_{n}}{h_{n}}\right)^{2(q-1)}\left\|(U^{n+1})^{-}\right\|^{2},
	\end{equation*}
	and hence we get
	\begin{equation*}
	\sum_{j=1}^{N_{n}}{\left((u_{j}^{n+1})^{-}-(u_{j}^{n})^{-}\right)(u_{j}^{n+1})^{-}}\leq 2\tau_{n}\left(\dfrac{M_{n}}{h_{n}}\right)^{2(q-1)}\left\|(U^{n+1})^{-}\right\|^{2}.
	\end{equation*}
	We use
	\begin{equation*}
	(a-b)a=\dfrac{a^{2}}{2}-\dfrac{b^{2}}{2}+\dfrac{(a-b)^{2}}{2},
	\end{equation*}
	then we have
	\begin{equation*}
	\dfrac{1}{2}\sum_{j=1}^{N_{n}}{(u_{j}^{n+1})^{-}}^{2}-\dfrac{1}{2}\sum_{j=1}^{N_{n}}{(u_{j}^{n})^{-}}^{2}+\dfrac{1}{2}\sum_{j=1}^{N_{n}}{((u_{j}^{n+1})^{-}-(u_{j}^{n})^{-})^{2}}\leq 2\tau_{n}\left(\dfrac{M_{n}}{h_{n}}\right)^{2(q-1)}\left\|(U^{n+1})^{-}\right\|^{2}
	\end{equation*}
	which implies
	\begin{equation*}
	\left(1-4\tau_{n}\left(\dfrac{M_{n}}{h_{n}}\right)^{2(q-1)}\right)\left\|(U^{n+1})^{-}\right\|^{2}\leq \left\|(U^{n})^{-}\right\|^{2}.
	\end{equation*}
	We use that $\lambda =\dfrac{\tau}{h^{2}}<\dfrac{1}{16}$, we can verify that $$1-4\tau_{n}\left(\dfrac{M_{n}}{h_{n}}\right)^{2(q-1)}>0$$ 
	and finally we have 
	\begin{equation*}
	\left\|(U^{n+1})^{-}\right\|^{2}\leq \dfrac{1}{1-4\tau_{n}\left(\dfrac{M_{n}}{h_{n}}\right)^{2(q-1)}}\left\|(U^{n})^{-}\right\|^{2}.
	\end{equation*}
	By $U^{n}\geq 0$, we have $(U^{n})^{-}=0$, which implies that $(U^{n+1})^{-}=0$, this gives $U^{n+1}\geq 0.$
\end{proof}
\subsection{Monotony: }
The following result, analogue of lemma \ref{monotonieexacte}, establishes monotony for the difference solution:
\begin{lem}
	Under the assumptions \textbf{(A1)-(A5)}, let $U^{n}$ be the solution of \eqref{approchee} and $m=\dfrac{N_{n}+1}{2}$, then we have
	\begin{equation*}
	\left\{
	\begin{array}{lll}
	0<u_{j}^{n}<u_{j+1}^{n} \text{ for } j=1,...,m-1 \text{ and } n\geq 0 \\
	0<u_{j+1}^{n}<u_{j}^{n} \text{ for } j=m,...,N_{n} \text{ and } n\geq 0 \\
	\end{array}
	\right.
	\end{equation*}
\end{lem}
\begin{proof}
	We will prove monotony by applying the similar argument as the nonnegativity to $v_{j}^{n}=u_{j+1}^{n}-u_{j}^{n}$ for $n \geq 0$ and $1\leq j \leq m-1$. Let 
	\begin{equation*}
	V^{n}=(v_{1}^{n},v_{2}^{n},...,v_{m-1}^{n})^{t}.
	\end{equation*}
	It is easy to see that for $j=1,...,m-1,\ v_{j}^{n}$ satisfies
	\begin{align}
	 \label{afef}
	\dfrac{v_{j}^{n+1}-v_{j}^{n}}{\tau_{n}}&=\dfrac{v_{j+1}^{n+1}-2v_{j}^{n+1}+v_{j-1}^{n+1}}{h_{n}^{2}}+(u_{j+1}^{n})^{p}-(u_{j}^{n})^{p}\\
	 &-\dfrac{1}{(2h_{n})^{q}}\left(\left|v_{j+1}^{n}-v_{j}^{n}\right|^{q-1}\left|v_{j+1}^{n+1}-v_{j}^{n+1}\right|-\left|v_{j}^{n}-v_{j-1}^{n}\right|^{q-1}\left|v_{j}^{n+1}-v_{j-1}^{n+1}\right|\right).\nonumber
	\end{align}
	In view of assumption \textbf{(A3)}, we see that $V^{n}\geq 0$ holds for $n=0$. Supposing that it holds for some $n\geq 0$, we have to show that $V^{n+1}\geq 0$. We use that
	\begin{equation*}
	(u_{j+1}^{n})^{p}-(u_{j}^{n})^{p}>0 \text{ and } \dfrac{1}{(2h_{n})^{q}}\left(\left|v_{j}^{n}-v_{j-1}^{n}\right|^{q-1}\left|v_{j}^{n+1}-v_{j-1}^{n+1}\right|\right)>0.
	\end{equation*} 
	We multiply equation \eqref{afef} by $(v_{j}^{n+1})^{-}$ and we sum for $j=1,...,m-1$, we obtain
	\begin{eqnarray*}
		&& \sum_{j=1}^{m-1}{\frac{(v_{j}^{n+1})^{-}-(v_{j}^{n})^{-}}{\tau_{n}}(v_{j}^{n+1})^{-}}
		+\frac{1}{h_{n}^{2}}\sum_{j=1}^{m-1}{(v_{j+1}^{n+1})^{+}(v_{j}^{n+1})^{-}}\\
		&& \ \ \ \ -\frac{1}{h_{n}}\sum_{j=1}^{m-1}{{\frac{(v_{j+1}^{n+1})^{-}-(v_{j}^{n+1})^{-}}{h_{n}}(v_{j}^{n+1})^{-}}}+\frac{1}{h_{n}^{2}}\sum_{j=1}^{m-1}{(v_{j-1}^{n+1})^{+}(v_{j}^{n+1})^{-}}\\
		&& \ \ \ \ +\frac{1}{h_{n}}\sum_{j=1}^{m-1}{\frac{(v_{j}^{n+1})^{-}-(v_{j-1}^{n+1})^{-}}{h_{n}}(v_{j}^{n+1})^{-}}
		- \sum_{j=1}^{m-1}\left|\frac{v_{j+1}^{n}-v_{j}^{n}}{2h_{n}}\right|^{q-1}\left|\frac{v_{j+1}^{n+1}-v_{j}^{n+1}}{2h_{n}}\right|(v_{j}^{n+1})^{-}\\
		&=&-\sum_{j=1}^{m-1}(v_{j}^{n})^{+}(v_{j}^{n+1})^{-}-\sum_{j=1}^{m-1}((u_{j+1}^{n})^{p}-(u_{j}^{n})^{p})(v_{j}^{n+1})^{-}\\
		&&\ \ \ \ -\sum_{j=1}^{m-1}\left|\frac{v_{j+1}^{n}-v_{j}^{n}}{2h_{n}}\right|^{q-1}\left|\frac{v_{j+1}^{n+1}-v_{j}^{n+1}}{2h_{n}}\right|(v_{j}^{n+1})^{-}\\
		&\leq & 0,
	\end{eqnarray*}
	so
	\begin{eqnarray*}
		&& \sum_{j=1}^{m-1}{\frac{(v_{j}^{n+1})^{-}-(v_{j}^{n})^{-}}{\tau_{n}}(v_{j}^{n+1})^{-}}
		+\frac{1}{h_{n}^{2}}\sum_{j=1}^{m-1}{(v_{j+1}^{n+1})^{+}(v_{j}^{n+1})^{-}}\\
		&&\ \ \ \ -\frac{1}{h_{n}}\sum_{j=1}^{m-1}{{\frac{(v_{j+1}^{n+1})^{-}-(v_{j}^{n+1})^{-}}{h_{n}}(v_{j}^{n+1})^{-}}}+\frac{1}{h_{n}^{2}}\sum_{j=1}^{m-1}{(v_{j-1}^{n+1})^{+}(v_{j}^{n+1})^{-}}\\
		&& \ \ \ \ +\frac{1}{h_{n}}\sum_{j=1}^{m-1}{\frac{(v_{j}^{n+1})^{-}-(v_{j-1}^{n+1})^{-}}{h_{n}}(v_{j}^{n+1})^{-}} \\
		& \leq & \dfrac{1}{2}\left(\frac{M_{n}}{h_{n}}\right)^{q-1}\sum_{j=1}^{m-1}\left|\frac{v_{j+1}^{n+1}-v_{j}^{n+1}}{h_{n}}\right|(v_{j}^{n+1})^{-}.
	\end{eqnarray*}
	We use the same calculations as the proof of lemma \ref{positiviteexacte}, we see that
	\begin{eqnarray*}
		& \displaystyle\sum\limits_{j=1}^{m-1}{\dfrac{(v_{j}^{n+1})^{-}-(v_{j}^{n})^{-}}{\tau_{n}}(v_{j}^{n+1})^{-}}
		+\sum\limits_{j=1}^{m-1}{\left(\dfrac{(v_{j+1}^{n+1})^{-}-(v_{j}^{n+1})^{-}}{h_{n}}\right)^{2}}\\
		& \leq \dfrac{1}{2}\left(\dfrac{M_{n}}{h_{n}}\right)^{q-1}\displaystyle\sum\limits_{j=1}^{m-1}\left|\dfrac{(v_{j+1}^{n+1})^{-}-(v_{j}^{n+1})^{-}}{h_{n}}\right|(v_{j}^{n+1})^{-}
		+\dfrac{((v_{m}^{n+1})^{-})^{2}}{h_{n}}.
	\end{eqnarray*} 
	We denote by $\left\|X\right\|_{*}:=\sum\limits_{j=1}^{m-1}(X_{j})^{2}$ and for all $\epsilon >0$ we have
	\begin{eqnarray*}
		&& \sum_{j=1}^{m-1}{\frac{(v_{j}^{n+1})^{-}-(v_{j}^{n})^{-}}{\tau_{n}}(v_{j}^{n+1})^{-}}
		+\left\|D(V^{n+1})^{-}\right\|_{*}^{2}\\
		& \leq& \epsilon \left\|D(V^{n+1})^{-}\right\|_{*}^{2} 
		+\frac{1}{4\epsilon}\left(\frac{M_{n}}{h_{n}}\right)^{2(q-1)}\left\|(V^{n+1})^{-}\right\|_{*}^{2}
		+\frac{1}{h_{n}}\left\|(V^{n+1})^{-}\right\|_{*}^{2}.
	\end{eqnarray*} 
	If we take $\epsilon =\dfrac{1}{2}$ we obtain
	\begin{eqnarray*}
		&& \sum_{j=1}^{m-1}{\frac{(v_{j}^{n+1})^{-}-(v_{j}^{n})^{-}}{\tau_{n}}(v_{j}^{n+1})^{-}}
		+\left\|D(V^{n+1})^{-}\right\|_{*}^{2}\\
		& \leq& \frac{1}{2} \left\|D(V^{n+1})^{-}\right\|_{*}^{2} 
		+\frac{1}{2}\left(\frac{M_{n}}{h_{n}}\right)^{2(q-1)}\left\|(V^{n+1})^{-}\right\|_{*}^{2}
		+\frac{1}{h_{n}}\left\|(V^{n+1})^{-}\right\|_{*}^{2}.
	\end{eqnarray*}
	Then we can deduce that
	\begin{equation*}
	\sum_{j=1}^{m-1}{((v_{j}^{n+1})^{-}-(v_{j}^{n})^{-})(v_{j}^{n+1})^{-}}
	\leq \tau_{n}\left(\frac{1}{h_{n}}+\frac{1}{2}\left(\frac{M_{n}}{h_{n}}\right)^{2(q-1)}\right)\left\|(V^{n+1})^{-}\right\|_{*}^{2},
	\end{equation*}
	which implies that
	\begin{equation*} \left\|(V^{n+1})^{-}\right\|_{*}^{2}-\left\|(V^{n})^{-}\right\|_{*}^{2}+\sum_{j=1}^{m-1}((v_{j}^{n+1})^{-}-(v_{j}^{n})^{-})^{2} \leq 2\tau_{n}\left(\frac{1}{h_{n}}+\frac{1}{2}\left(\frac{M_{n}}{h_{n}}\right)^{2(q-1)}\right)\left\|(V^{n+1})^{-}\right\|_{*}^{2}.
	\end{equation*}
	And hence we get
	\begin{equation*}
	\left(1-2\tau_{n}\left(\frac{1}{h_{n}}+\frac{1}{2}\left(\frac{M_{n}}{h_{n}}\right)^{2(q-1)}\right)\right)\left\|(V^{n+1})^{-}\right\|_{*}^{2}\leq \left\|(V^{n})^{-}\right\|_{*}^{2}.
	\end{equation*}
	We use that $\lambda =\dfrac{\tau}{h^{2}}<\dfrac{1}{16}$ we can verify that $$1-2\tau_{n}\left(\dfrac{1}{h_{n}}+\dfrac{1}{2}\left(\dfrac{M_{n}}{h_{n}}\right)^{2(q-1)}\right)>0.$$ 
	And finally by $V^{n}\geq 0$, we can deduce that $(V^{n+1})^{-}=0$, this gives $V^{n+1}\geq 0.$\\
	We do the same thing to obtain that $v_{j}^{n}=u_{j}^{n}-u_{j+1}^{n}\geq 0$ for $n\geq 0$ and $m\leq j\leq N_{n}.$
\end{proof}
\subsection{Symmetry:}
The last property of the difference solution is the symmetry, analogue of lemma \ref{symetrieexacte}.
\begin{lem}
	Under the assumption \textbf{(A1)-(A5)}, for $U^{n}$ solution of \eqref{approchee} we have
	\begin{equation*}
	u_{m-i}^{n}=u_{m+i}^{n} \text{ for all } i=1,...,m-1 \text{ and } n\geq 0.
	\end{equation*}
	\label{symetrienum}
\end{lem}
\begin{proof}
	For $n\geq 0$, let $\lambda_{n}:=\dfrac{\tau_{n}}{h_{n}^{2}}$. For $j=1,...,N_{n}$ the first equation of \eqref{approchee} can be rewritten as
	\begin{equation}
	-\lambda_{n}u_{j-1}^{n+1}+(1+2\lambda_{n})u_{j}^{n+1}-\lambda_{n}u_{j+1}^{n+1}+\frac{\tau_{n}}{(2h_{n})^{q}}\left|u_{j+1}^{n}-u_{j-1}^{n}\right|^{q-1}\left|u_{j+1}^{n+1}-u_{j-1}^{n+1}\right|=u_{j}^{n}+\tau_{n}(u_{j}^{n})^{p}.
	\label{matricedesymetrie}
	\end{equation}
	In view of the assumption \textbf{(A2)}, we see that $u_{m-i}^{0}=u_{m+i}^{0}$ for all $i=1,...,m-1$. Supposing that it holds for some $n\geq 0$, then for $i=1,...,m-1,$
	\begin{eqnarray}
	u_{m-i}^{n}&=&u_{m+i}^{n}.
	\label{eq22}
	\end{eqnarray}
	We have to show that $u_{m-i}^{n+1}=u_{m+i}^{n+1}$ for $i=1,...,m-1$.\\
	We use \eqref{matricedesymetrie}, \eqref{eq22}, positivity and monotony we get the system $(S)$
			\begin{equation*} 
			\left\lbrace
			\begin{array}{lcl} 
			\ \ \ \ \ \ \ \ \ \ \ \ \ \ \ \ \ \ \ \ \ \ \ \ \ \ \ \ \ (1+2\lambda_{n})u_{1}^{n+1}\ \ +\left(-\lambda_{n}+\alpha_{1}^{n}\right)u_{2}^{n+1}&=& u_{1}^{n}+\tau_{n}(u_{1}^{n})^{p}.\\ 
			\left(-\lambda_{n}-\alpha_{i}^{n}\right)u_{i-1}^{n+1}\ \ \ \ +(1+2\lambda_{n})u_{i}^{n+1}\ \ +	\left(-\lambda_{n}+\alpha_{i}^{n}\right)u_{i+1}^{n+1}&=& u_{i}^{n}+\tau_{n}(u_{i}^{n})^{p}, \  i=2,...,m-1,\\
			-\lambda_{n}u_{m-1}^{n+1} \ \ \ \ \ \ \ \ \ \ \ \ \  +(1+2\lambda_{n})u_{m}^{n+1}\ \ -\lambda_{n}u_{m+1}^{n+1}&=& u_{m}^{n}+\tau_{n}(u_{m}^{n})^{p}.\\ 
			\left(-\lambda_{n}+\alpha_{i}^{n}\right)u_{i-1}^{n+1}\ \ \ \ +(1+2\lambda_{n})u_{i}^{n+1}\ \ +	\left(-\lambda_{n}-\alpha_{i}^{n}\right)u_{i+1}^{n+1}&=& u_{i}^{n}+\tau_{n}(u_{i}^{n})^{p},\ i=m+1,...,N_{n}-1,\\
			\left(-\lambda_{n}+\alpha_{N_{n}}^{n}\right)u_{N_{n}-1}^{n+1}+(1+2\lambda_{n})u_{N_{n}}^{n+1}&=& u_{N_{n}}^{n}+\tau_{n}(u_{N_{n}}^{n})^{p}.\\	
			\end{array}
			\right. 
			\end{equation*}
where $\alpha^{n}$ such that $\alpha_{i}^{n}=\dfrac{\tau_{n}}{(2h_{n})^{q}}\left|u_{i+1}^{n}-u_{i-1}^{n}\right|^{q-1}, \ \text{for } i=1,...,N_{n}$ which is symmetric because of \eqref{eq22}.\\ 
Using positivity, monotony and \eqref{eq22}, then $(S)$ can be rewritten as	
\begin{equation}
QU^{n+1}=V^{n},
\label{inversible}
\end{equation}
	where 	$V^{n}$ is a symmetric vector defined by
	$V^{n}=(v_{1}^{n},...,v_{m-1}^{n},v_{m}^{n},v_{m-1}^{n},...,v_{1}^{n})^{t}$
	with 
	\begin{equation*}
	v_{i}^{n}=u_{i}^{n}+\tau_{n}(u_{i}^{n})^{p} \text{ for } 1\leq i\leq m.
	\end{equation*}
And $Q$ is an $N_{n}\times N_{n}$ tridiagonal matrix such that $Q=(q_{ij})_{i,j}$ with
	\begin{equation*}
	\left. 
	\begin{array}{lll}
	q_{ij}&=0\ \ for \ \ \left|i-j\right|>1,\\
	q_{ii}&=1+2\lambda_{n}\ \ for \ \ i=1,...,N_{n},\\
	q_{i+1,i}&=\left\lbrace  \begin{array}{l}
	-\lambda_{n}-\alpha_{i+1}^{n}\ \ for \ \ i=1,...,m-1,\\
		-\lambda_{n}+\alpha_{i+1}^{n}\ \ for \ \ i=m,...,N_{n}-1,\\
			\end{array}
			\right.\\
	q_{i,i+1}&=\left\lbrace \begin{array}{l}
			-\lambda_{n}+\alpha_{i}^{n}\ \ for \ \ i=1,...,m,\\
			-\lambda_{n}-\alpha_{i}^{n}\ \ for \ \ i=m+1,...,N_{n}-1,\\
				\end{array}
				\right. 
		\end{array}
		\right. 
	\end{equation*}
	We show that $Q$ is a strictly diagonal-dominant real matrix, in fact :\\
	For all $i=1,...,N_{n}$ we have 
	\begin{equation*}
	\left|q_{ii}\right|=1+2\lambda_{n}\ \ \text{ and} \ \ \ \sum\limits_{i\neq j}\left|q_{ij}\right|=\left|-\lambda_{n}+\alpha_{i}^{n}\right|+\lambda_{n}+\alpha_{i}^{n}
	\end{equation*}
	Using \eqref{hypothesepositive}, we get $\dfrac{\tau_{n}M_{n}^{q-1}}{2h_{n}^{q}}\leq \dfrac{\tau_{n}}{h_{n}^{2}}= \lambda_{n}$, and then
			\begin{equation*}
			\alpha_{i}^{n}\leq \dfrac{\tau_{n}M_{n}^{q-1}}{2h_{n}^{q}}\leq \lambda_{n},
			\end{equation*}	
			which implies that
			\begin{equation*}
			\sum\limits_{i\neq j}\left|q_{ij}\right|=2\lambda_{n}<1+2\lambda_{n}=\left|q_{ii}\right|.
			\end{equation*}
Using the Hadamard lemma we deduce that $Q$ is invertible. Hence \eqref{inversible} has a unique solution.\\
	Now, it is easy to see that if
	\begin{align*}
	U^{n+1}=\left(
	\begin{array}{ccccccccccc}
	u_{1}^{n+1} &\cdots u_{m-1}^{n+1} &u_{m}^{n+1}&u_{m+1}^{n+1} &\cdots &u_{N_{n}}^{n+1} 
	\end{array}
	\right)^{t}
	\end{align*}
	is a solution of \eqref{inversible} and using the symmetry of $V^{n}$ we obtain that
	\begin{align*}
	W^{n+1}=
	\left(
	\begin{array}{ccccccccccc}
	u_{1}^{n+1} &\cdots u_{m-1}^{n+1} &u_{m}^{n+1}&u_{m-1}^{n+1} &\cdots &u_{1}^{n+1} 
	\end{array}
	\right)^{t}
	\end{align*} 
	is also a solution of \eqref{inversible}. By uniqueness, we deduce that $U^{n+1}=W^{n+1}$ which achieves the proof of symmetry.
\end{proof}
\section{Blow up theorem }
In this section we will prove that the solution of the numerical problem blows up for all $p>1$ and $1\leq q \leq \frac{2p}{p+1}.$
\begin{th1}
	We suppose that the initial data satisfies \textbf{(A1)-(A5)}, then the solution of \eqref{approchee} blows up and we have 
	\begin{equation*}
	\lim_{n\rightarrow +\infty} u_{m}^{n}=+\infty.
	\end{equation*}
	\label{explosionnumerique}
\end{th1}
To prove the theorem, we need the next lemma:
\begin{lem}
	For a large initial data, we have $u_{m}^{n}>> 1$ for all $n\geq 0$. Moreover we have, $u_{m}^{n+1}\geq u_{m}^{n}$  for all $n\geq 0.$
	\label{lemcroissance}
\end{lem}
\begin{proof}
	For $n=0$, we have $u_{m}^{0}>>1$ because of \textbf{(A4)}. Supposing that it holds for some $n \geq 0$, we have to show that $u_{m}^{n+1}>> 1$. In the first equation of \eqref{approchee}, if we take $j=m$ and we use symmetry we obtain
	\begin{equation*}
	(1+2\lambda_{n})u_{m}^{n+1}=u_{m}^{n}+2\lambda_{n}u_{m-1}^{n+1}+\tau_{n}(u_{m}^{n})^{p},
	\end{equation*}
	where $\lambda_{n}:=\dfrac{\tau_{n}}{h_{n}^{2}}$, and then we have
	\begin{equation}
	u_{m}^{n+1} \geq \dfrac{1+\tau_{n}(u_{m}^{n})^{p-1}}{1+2\lambda_{n}}u_{m}^{n}.
	\label{eq41}
	\end{equation}
	Using the recurrence hypothesis we get
	\begin{equation*}
	1+\tau_{n}(u_{m}^{n})^{p-1}=1+\tau \ \  \text { and } \ \ \  \lambda_{n}=\dfrac{\tau}{2^{\frac{2}{2-q}}(u_{m}^{n})^{p-\frac{q}{2-q}}}.
	\end{equation*}
	Then 
	\begin{eqnarray*}
		\dfrac{1+\tau_{n}(u_{m}^{n})^{p-1}}{1+2\lambda_{n}}&=&\dfrac{1+\tau}{1+\tau 2^{1-\frac{2}{2-q}}(u_{m}^{n})^{-p+\frac{q}{2-q}}}\\
		&=&\dfrac{1+\tau}{1+\tau 2^{\frac{-q}{2-q}}(u_{m}^{n})^{\frac{-2p+q(p+1)}{2-q}}}.
	\end{eqnarray*}
	For $1\leq q \leq \dfrac{2p}{p+1}$, we have
	\begin{equation}
	-r:=\dfrac{-2p+q(p+1)}{2-q}\leq 0,
	\label{eq1}
	\end{equation}
	which implies that
	\begin{equation}
	u_{m}^{n+1}\geq \dfrac{1+\tau}{1+\tau(2^{\frac{-q}{2-q}}(u_{m}^{n})^{-r})}u_{m}^{n}.
	\label{eq2}
	\end{equation}
	Now we have to show that
	\begin{equation*}
	\dfrac{1+\tau}{1+\tau(2^{\frac{-q}{2-q}}(u_{m}^{n})^{-r})}\geq 1.
	\end{equation*}
	Using $u_{m}^{n}>>1$ and $r\geq 0$, we have
	\begin{equation*}
	2^{\frac{-q}{2-q}}(u_{m}^{n})^{-r}=\dfrac{1}{2^{\frac{q}{2-q}}(u_{m}^{n})^{r}} <1.
	\end{equation*}
	Then 
	\begin{equation*}
	1+\tau> 1+\tau 2^{\frac{-q}{2-q}}(u_{m}^{n})^{-r},
	\end{equation*}
	hence 
	\begin{equation}
	\dfrac{1+\tau}{1+\tau(2^{\frac{-q}{2-q}}(u_{m}^{n})^{-r})}> 1,
	\label{um}
	\end{equation}
	and then, by $u_{m}^{n}>>1$
	\begin{equation*}
	u_{m}^{n+1} \geq \dfrac{1+\tau}{1+\tau(2^{\frac{-q}{2-q}}(u_{m}^{n})^{-r})}u_{m}^{n}>>1.
	\end{equation*}
	Moreover, for all $n\geq 0$ we deduce from \eqref{eq2} and \eqref{um} that $u_{m}^{n+1}\geq u_{m}^{n}$ which proves lemma \ref{lemcroissance}.
\end{proof}
Now we can prove theorem \ref{explosionnumerique}.
\begin{proof}
	Using lemma \ref{lemcroissance}, \eqref{eq1} and \eqref{eq2}, we can write that:
	\begin{eqnarray*}
		u_{m}^{n+1}& \geq &\dfrac{1+\tau}{1+\tau 2^{\frac{-q}{2-q}}(u_{m}^{n})^{\frac{-2p+q(1+p)}{2-q}}}u_{m}^{n} \\
		& \geq & \dfrac{1+\tau}{1+\tau 2^{\frac{-q}{2-q}}(u_{m}^{0})^{\frac{-2p+q(1+p)}{2-q}}}u_{m}^{n}, 
	\end{eqnarray*}
	which implies by iterations
	\begin{equation*}
	u_{m}^{n}\geq \left(\dfrac{1+\tau}{1+\tau 2^{\frac{-q}{2-q}}(u_{m}^{0})^{\frac{-2p+q(1+p)}{2-q}}}\right)^{n}u_{m}^{0}.  
	\end{equation*}
	For a large initial data, we have 
	\begin{equation*}
	\dfrac{1+\tau}{1+\tau 2^{\frac{-q}{2-q}}(u_{m}^{0})^{\frac{-2p+q(1+p)}{2-q}}}>1,
	\end{equation*}
	this implies that $u_{m}^{n}\rightarrow +\infty$ as $n \rightarrow +\infty$, which achieve the proof of theorem \ref{explosionnumerique}.
\end{proof}
\section{Numerical simulation}
In this section, we present some numerical simulation (with Matlab) that illustrate our results.\\
As it is shown in Figure \ref{donneeinitiale}, we take $u_{0}(x)=10^{3}\sin(\frac{\pi}{2}(x+1))$, which satisfies the conditions \textbf{(A1)-(A5)}.
\begin{figure}[H]
	\centering
	\includegraphics[width=2.8in,height=2.8in]{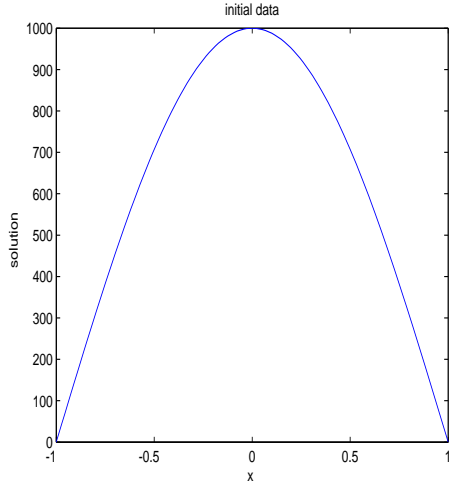}
	\caption{initial data: $u_{0}(x)=1000\sin\left(\dfrac{\pi}{2}(x+1)\right)$}
	\label{donneeinitiale}
\end{figure}
Next, we take $p=3$, $q=1.3 < \frac{2p}{p+1}$. Figures \ref{50iter}, \ref{200iter} and \ref{300iter} show the evolution of the numerical solution for different iterations. One can see that, numerically, the solution blows up in $x=0$.\\
\begin{figure}[H]
	\centering
	\includegraphics[width=3in,height=3in]{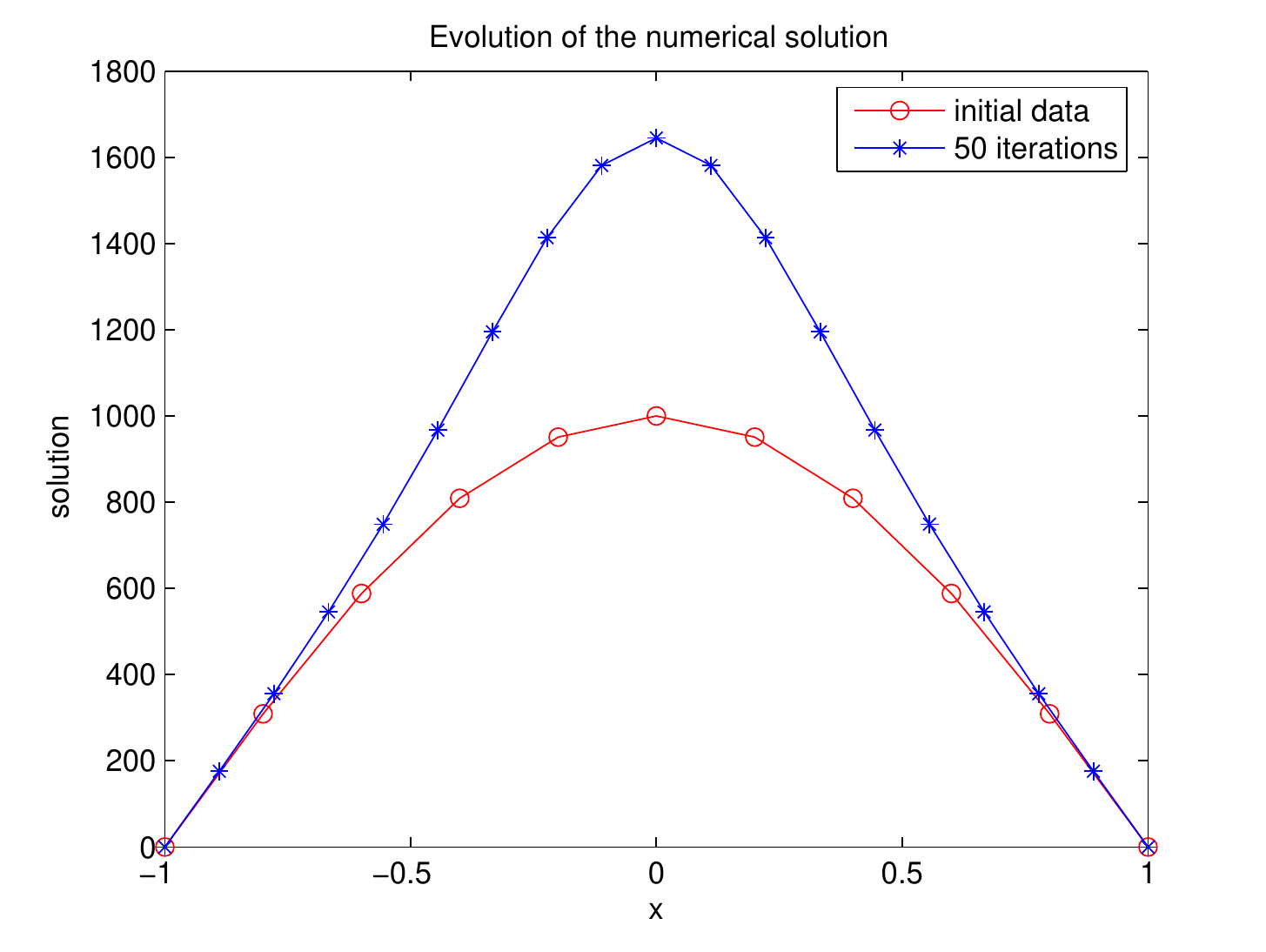}
	\caption{Evolution of the numerical solution in 50 iterations.}
	\label{50iter}
\end{figure}
\begin{figure}[H]
	\begin{minipage}[b]{0.40\linewidth}
		\centering
		\includegraphics[width=3in,height=3in]{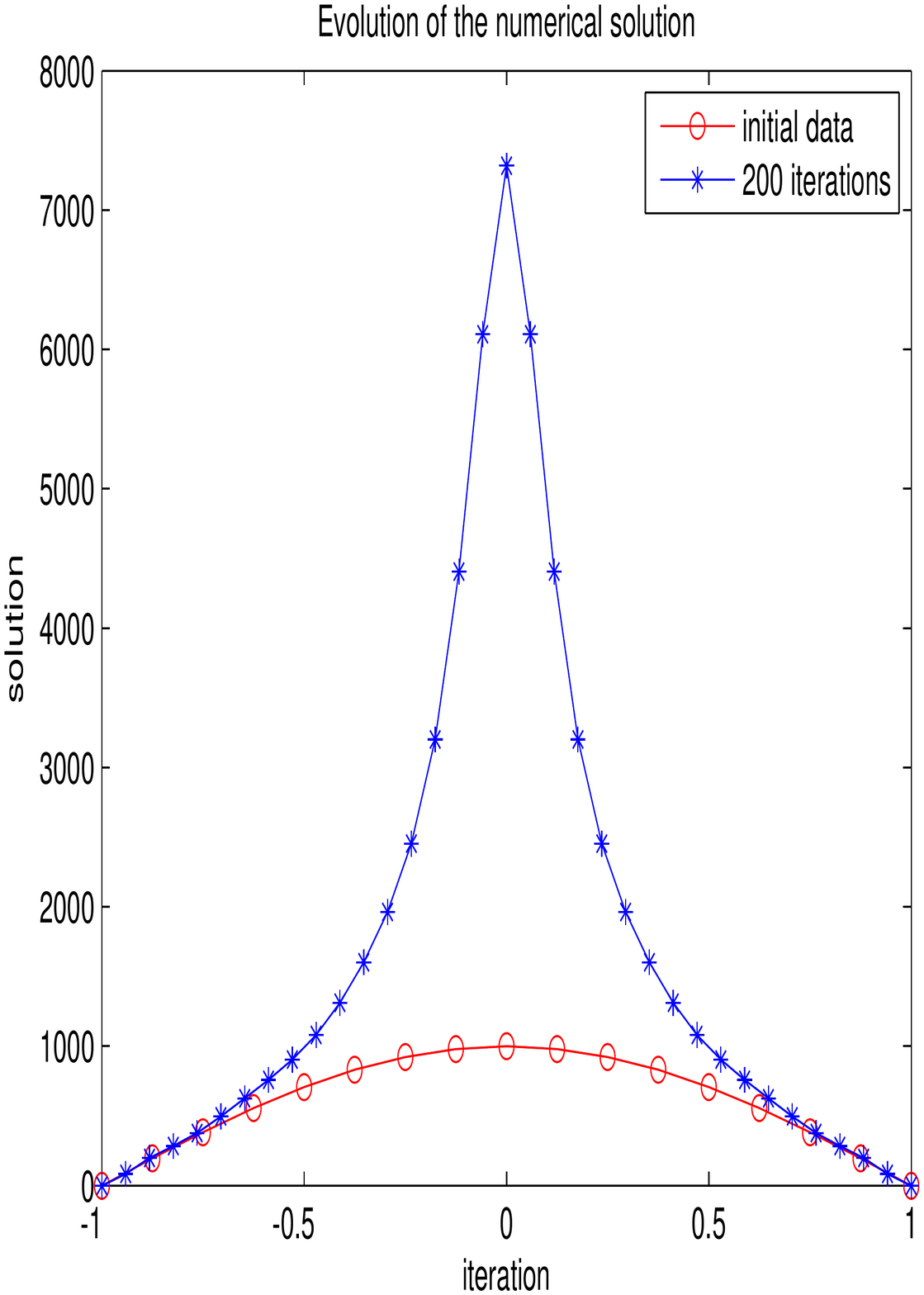}
		\caption{Evolution of the numerical solution in 200 iterations.}
		\label{200iter}
	\end{minipage}\hfill
	\begin{minipage}[b]{0.40\linewidth}
		\centering
		\includegraphics[width=3in,height=3in]{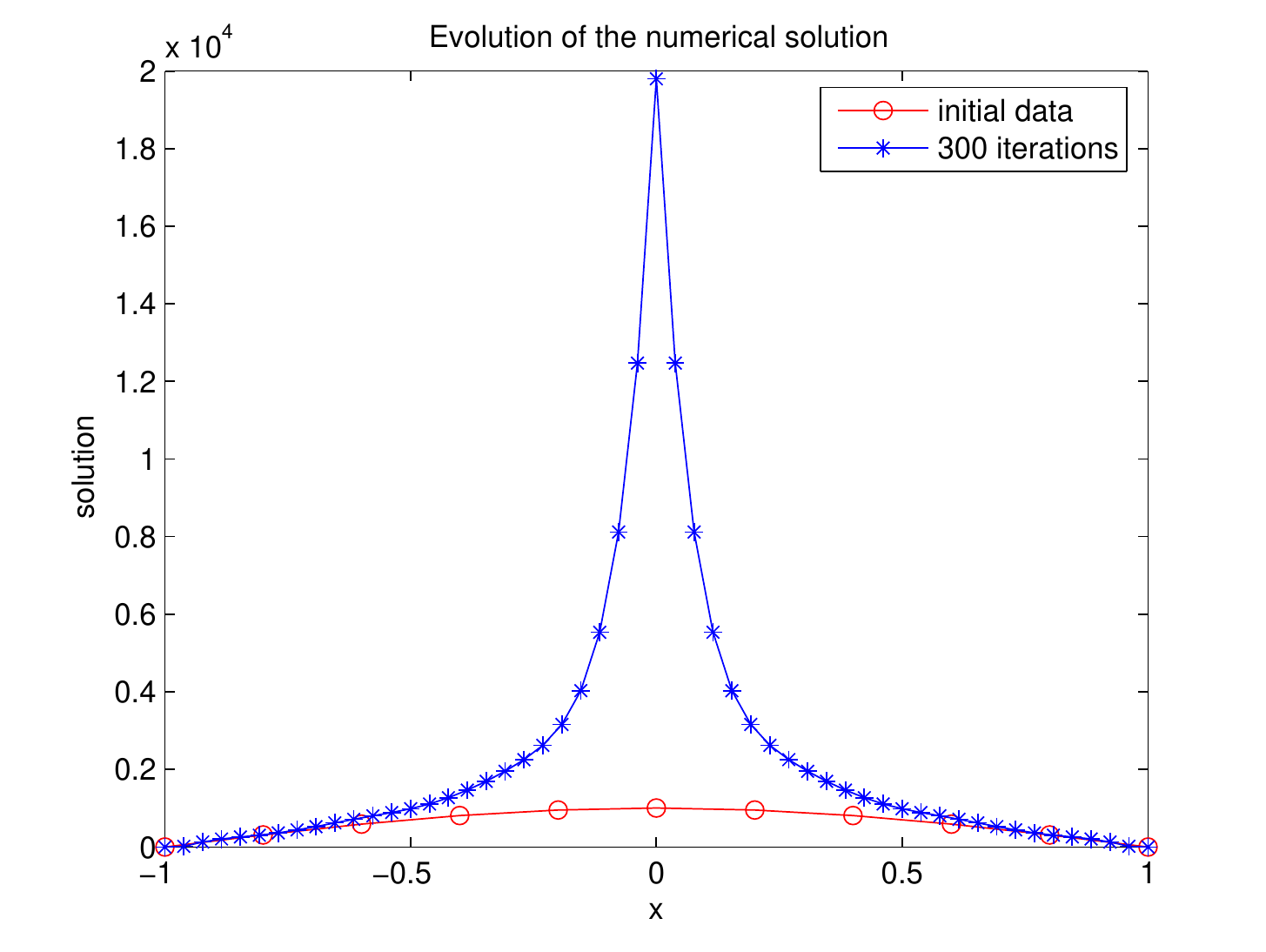}
		\caption{Evolution of the numerical solution in 300 iterations.}
		\label{300iter}
	\end{minipage}
\end{figure}
The growth of the solution leads to the reduction of $h_{n}$ and hence increasing the number of points of discretisation.
In Table \ref{tab}, we present some results about the decreasing of $h_{n}$ and the increasing of $N_{n}$ (the number of points of discretisation of the interval $[-1,1]$) in each iteration. Initially, simulation started with a discrete space step $h_{n}=0.138$, a discrete time step $\tau_{n}=10^{-4}$, a number of points of discretisation of the interval $[-1,1]$, $N_{n}=15$ and a maximum value $M_{n}=10^{3}$. After 350 iterations, we have an increase in the maximum value which leads to the decreasing of the discrete space step and discrete time step.
\begin{table}[H]
	\scriptsize
	\centering
	\begin{tabular}{|*{9}{c|}}
		\hline
		$Iteration$& $1$&$110$&$145$&$200$&$260$&$280$&$300$&$350 $\\ \hline
		$M_{n}$&$10^{3}$&$2.928.10^{3}$&$4.190.10^{3}$&$7.245.10^{3}$&$1.317.10^{4}$&$1.607.10^{4}$&$1.980.10^{4}$&$3.203.10^{4} $\\ \hline
		$\tau_{n}$&$10^{-4}$&$10^{-6}$&$5.10^{-7}$&$2.10^{-9}$&$2.10^{-10}$&$10^{-10}$&$4.10^{-11}$&$10^{-11}$ \\ \hline
		$h_{n}$&$0.13$&$0.87.10^{-1}$&$0.75.10^{-1}$&$0.59.10^{-1}$&$0.46.10^{-1}$&$0.42.10^{-1}$&$0.38.10^{-1}$&$0.31.10^{-1}$ \\ \hline
		$N_{n}$&$15$&$23$&$27$&$33$&$41$&$47$&$51$&$63$ \\ \hline
	\end{tabular}
	\caption{Reduction of $h_{n},\ \tau_{n}$ and increasing of the number of points of discretisation.}
	\label{tab}
\end{table}
From Figure \ref{fmax}, we observe the evolution of the numerical maximum point (blow-up point) $x=0$. It gives an idea about the blow up rate given in theorem \ref{tauxexplosion}. In fact, in Figure \ref{ftauxexplosion}, we have plot the function $h$ representing the theoretical blow up rate given in theorem \ref{tauxexplosion}:
\begin{equation*}
h(t)=\dfrac{C}{(T^{*}_{num}-t)^{\frac{1}{p-1}}},\ t\geq0
\end{equation*}
where $T^{*}_{num}$ is the blow up time estimated in \cite{hani} and $C$ is a constant such that for $t=0$ we have $h(0)=\left\|u_{0}\right\|_{\infty}=10^{3}$.
\begin{figure}[H]
	\begin{minipage}[b]{0.40\linewidth}
		\centering
		\includegraphics[width=3in,height=3in]{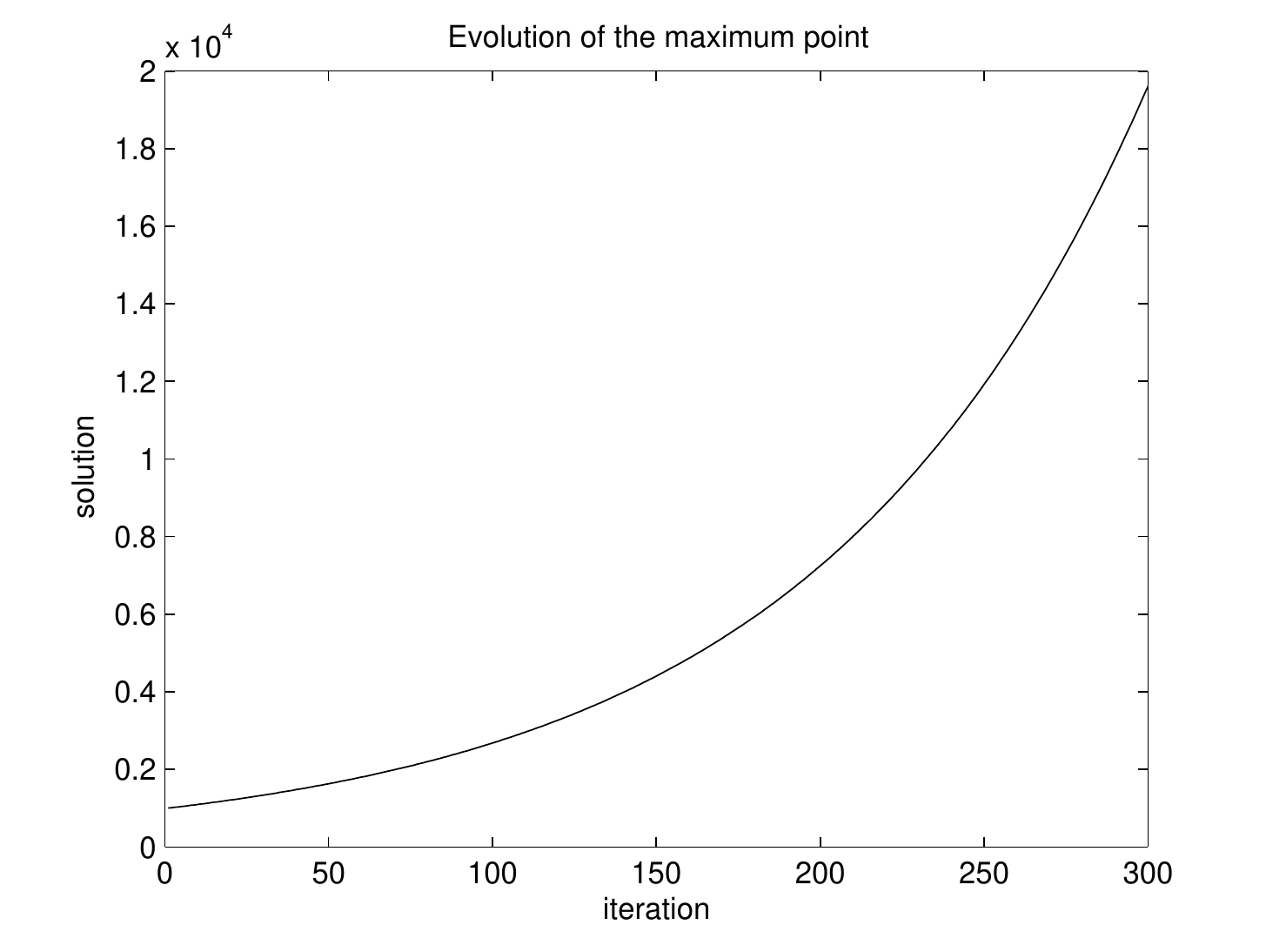}
		\caption{Evolution of the numerical maximum point (blow-up point) $x=0$ for $p=3$ and $q=1.3$.}
		\label{fmax}
	\end{minipage}\hfill
		\begin{minipage}[b]{0.40\linewidth}
			\centering
			\includegraphics[width=3in,height=3in]{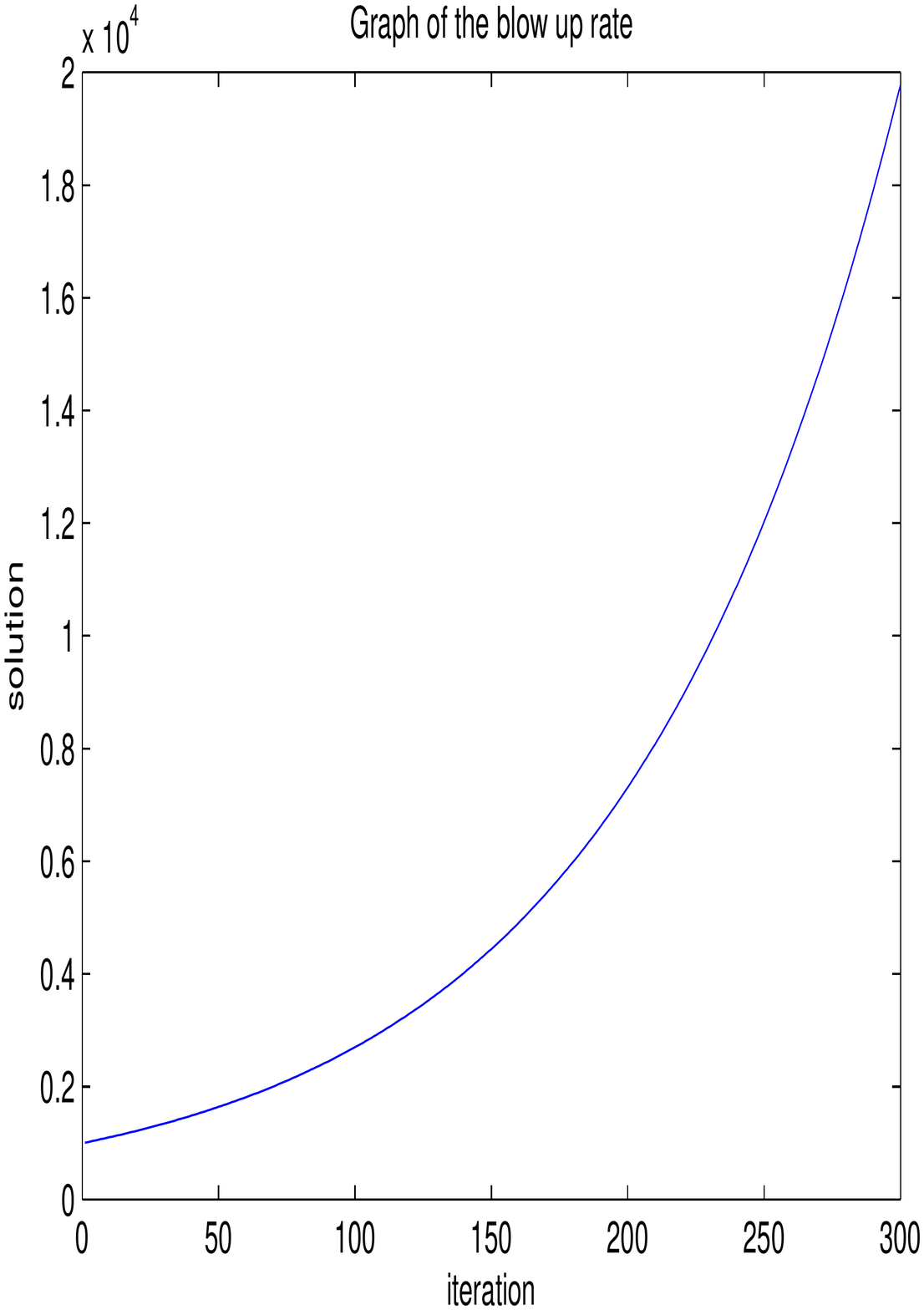}
			\caption{The blow up rate given by $ \left\|u(t)\right\|_{\infty}\approx h(t)$ with $T^{*}_{num}\approx 5.067.10^{-7}$, $C\approx 0.711481$ and $p=3$.}
			\label{ftauxexplosion}
		\end{minipage}
\end{figure}
We show in Figure \ref{small} that the solution decays with a small initial data $u_{0}(x)=\sin(\frac{\pi}{2}(x+1))$, and hence blowing up can not occur. This was proved theoretically in theorem \ref{petitedonnee}.
\begin{figure}[H]
	\centering
	\includegraphics[width=3in,height=3in]{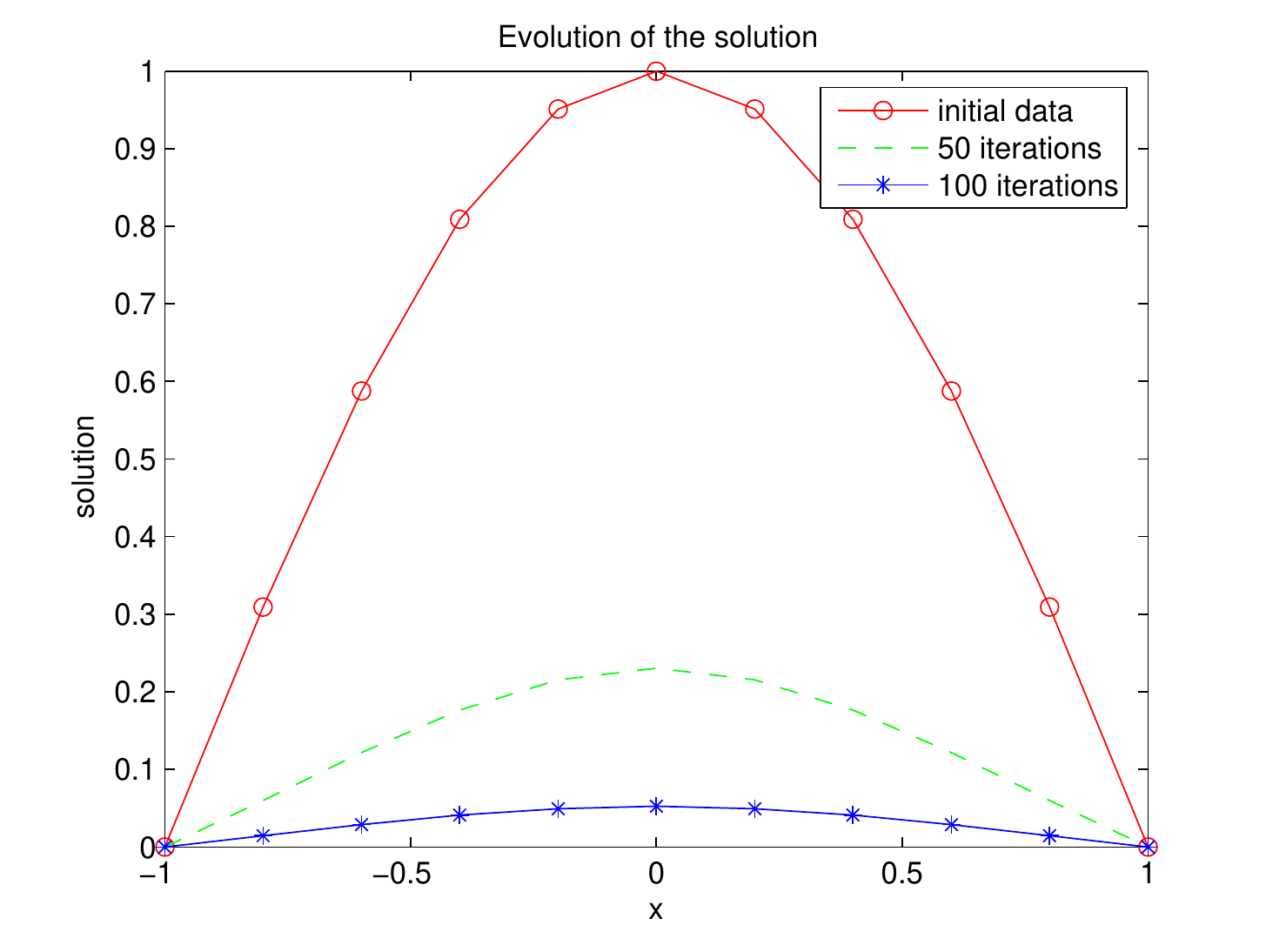}
	\caption{Global solution with a small initial data $u_{0}(x)=\sin(\frac{\pi}{2}(x+1)$.}
	\label{small}
\end{figure}
 We can see that the solution without gradient term shown in Figure \ref{2Dfujita} blows up more rapidly than the solution of the Chipot-Weissler equation shown in Figure \ref{2D}, this proves the damping effect of the gradient term.
\begin{figure}[H]
	\begin{minipage}[b]{0.40\linewidth}
		\centering
		\includegraphics[width=3in,height=3in]{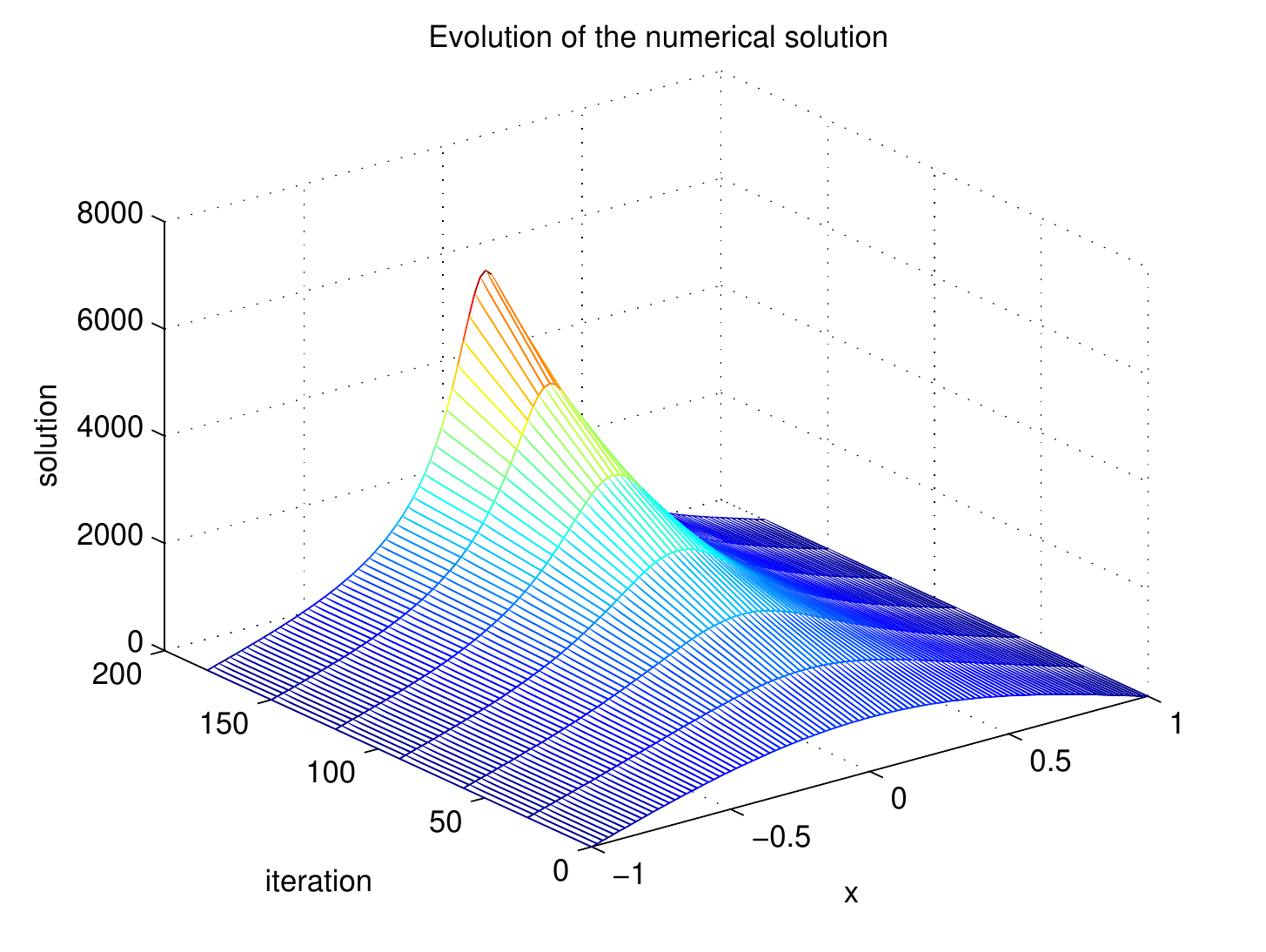}
		\caption{The shape of the numerical blow-up solution with gradient term.}
		\label{2D}
	\end{minipage}\hfill
	\begin{minipage}[b]{0.40\linewidth}
		\centering
		\includegraphics[width=3in,height=3in]{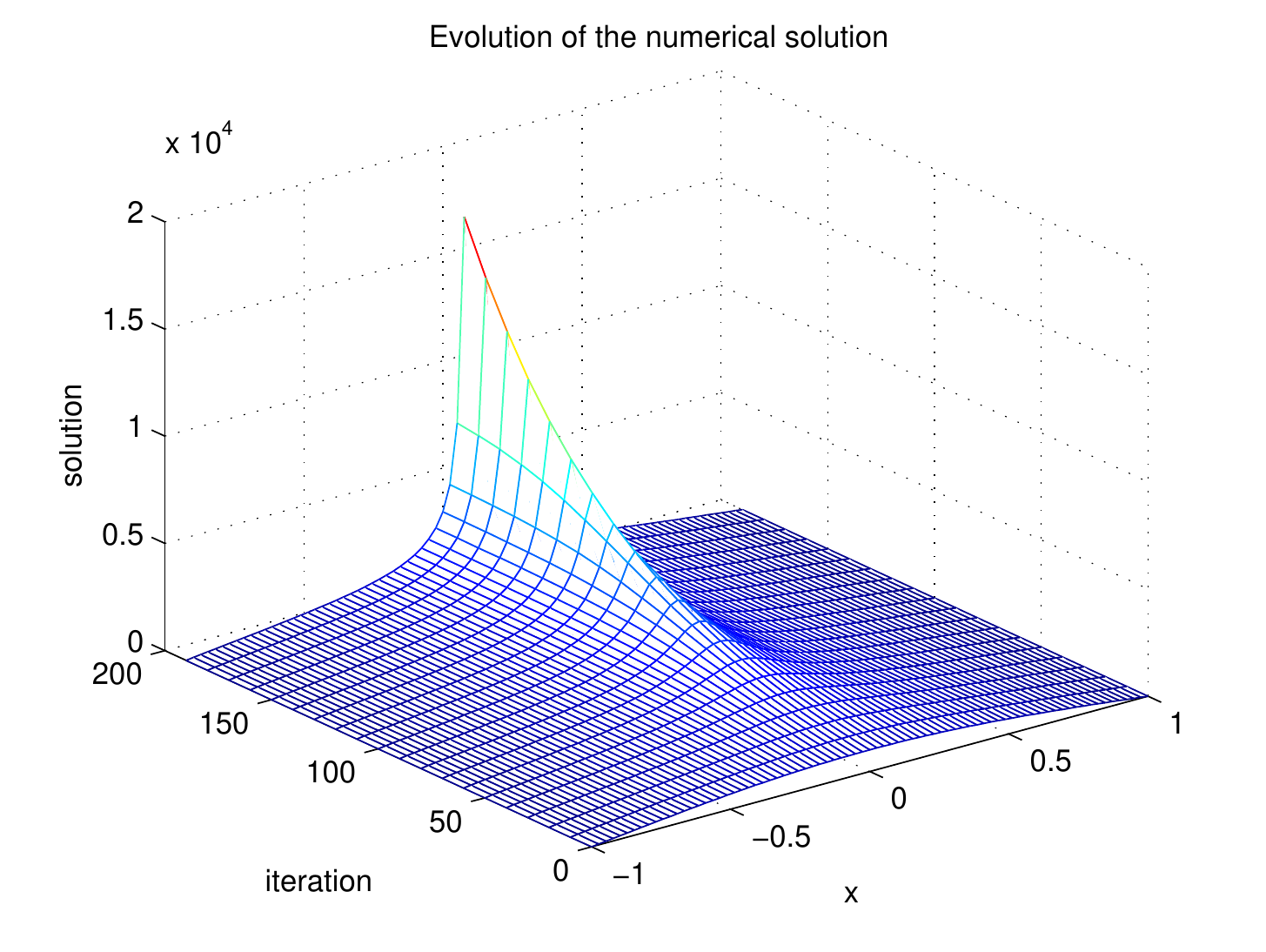}
		\caption{The shape of the numerical blow-up solution without gradient term.}
		\label{2Dfujita}
	\end{minipage}
\end{figure}

\section{Concluding remarks}
In this paper we have developped a numerical scheme in order to approximate the blow-up solution of the Chipot-Weissler equation. We have showed that we have blow up in $x_{m}=0.$\\
Our goal in another work is to use our scheme to study the competition between the gradient term which fights against blow up and the reaction term which may cause blow up in finite time as in the Fujita equation (without gradient term). In particular, we would like to answer some questions in the future study :
\begin{enumerate}
	\item Can we determine the numerical blow-up set exactly ?
	\item What can we say about the asymptotic behaviours of the numerical solution near the blow-up set?
	\item Can we give an approximation about the blow up time ?
	\item Let us consider the equation $u_{t}=\Delta u+a\left|u\right|^{p}-b\left|\nabla u\right|^{q}.$ What conditions should be satisfied by $a$ and $b$ to reproduce blowing-up phenomena ?
	\item Can we extend our results for $d\geq 2$ ?\\
\end{enumerate}
\begin{center}
	\textbf{Acknowledgment: }
\end{center}
The present paper is an outgrowth of the first author's thesis' under the guidance of the second author to him he is highly acknowledged.

\end{document}